\documentclass{amsart}
\pdfoutput=1

\numberwithin{equation}{subsection}
\let\oldsection\section
\renewcommand{\section}{
  \renewcommand{\theequation}{\thesection.\arabic{equation}}
  \oldsection}
\let\oldsubsection\subsection
\renewcommand{\subsection}{
  \renewcommand{\theequation}{\thesubsection.\arabic{equation}}
  \oldsubsection}
\usepackage[utf8]{inputenc}
\usepackage[T1]{fontenc}
\usepackage{lmodern}

\usepackage[margin=1.2in]{geometry}
\usepackage{setspace}
\usepackage[usenames, dvipsnames]{xcolor}
\usepackage{graphicx}
\usepackage{mathtools}
\usepackage{amssymb}
\usepackage{amsthm}

\usepackage{xspace}
\usepackage{tikz-cd}
\usepackage{tkz-euclide}
\usepackage{tkz-graph}
\usetikzlibrary{
  knots,
  hobby,
  decorations.pathreplacing,
  shapes.geometric,
  calc
}
\usepackage{bbm}
\usepackage{slashed}
\usepackage[backref=page, bookmarks=false]{hyperref}
\usepackage[capitalize, noabbrev]{cleveref}

\newsavebox{\pullback}
\sbox\pullback{
\begin{tikzpicture}
\draw (0,0) -- (1ex,0ex);
\draw (1ex,0ex) -- (1ex,1ex);
\end{tikzpicture}}

\DeclareMathOperator{\Hom}{Hom}

\newtheorem{lemma}[equation]{Lemma}
\newtheorem{corollary}[equation]{Corollary}
\newtheorem*{corollary*}{Corollary}
\newtheorem{proposition}[equation]{Proposition}

\newtheorem{theorem}[equation]{Theorem}
\newtheorem*{theorem*}{Theorem}

\newtheorem*{lemma*}{Lemma}
\newtheorem{maintheorem}{Theorem}[section]

\theoremstyle{definition}
\newtheorem{example}[equation]{Example}
\newtheorem*{example*}{Example}
\newtheorem{definition}[equation]{Definition}

\newtheorem{observation}[equation]{Observation}
\newtheorem{notation}[equation]{Notation}

\theoremstyle{remark}
\newtheorem{remark}[equation]{Remark}
\newtheorem{warning}[equation]{Warning}

\crefname{thm}{Theorem}{Theorems}
\crefname{lem}{Lemma}{Lemmas}
\crefname{cor}{Corollary}{Corollaries}
\crefname{prop}{Proposition}{Propositions}
\crefname{ex}{Exercise}{Exercises}
\crefname{exm}{Example}{Examples}
\crefname{defn}{Definition}{Definitions}
\crefname{claim}{Claim}{Claims}
\crefname{rem}{Remark}{Remarks}
\crefname{fct}{Fact}{Facts}
\crefname{note}{Note}{Notes}

\DeclarePairedDelimiter\paren{(}{)}

\makeatletter 
\def\l@subsection{\@tocline{2}{0pt}{1pc}{5pc}{}} \def\l@subsection{\@tocline{2}{0pt}{2pc}{6pc}{}} \makeatother

\makeatletter
	\let\oldparen\paren
	\def\paren{\@ifstar{\oldparen}{\oldparen*}}
\makeatother

\usepackage{microtype}
\usepackage{hypcap}

\hypersetup{
 colorlinks,
 linkcolor={red!50!black},
 citecolor={green!50!black},
 urlcolor={blue!80!black}
}

\newcommand{\id}{\mathrm{id}}

\newcommand{\cO}{\mathcal{O}}

\newcommand{\mbbN}{\mathbb{N}}

\newcommand{\cD}{\mcal{D}}

\newcommand{\Fun}{\mathrm{Fun}}

\newcommand{\Deltaop}{\Delta^{op}}

\newcommand{\Cat}{\mathrm{Cat}}

\newcommand{\Op}{\mathrm{Op}}
\newcommand{\un}{\mathrm{un}}

\newcommand{\Fin}{\mathrm{Fin}}

\newcommand{\Triv}{\mathrm{Triv}}

\newcommand{\Mul}{\mathrm{Mul}}

\newcommand{\Opun}{\Op^{\un}}

\newcommand{\Opunleq}[1]{\Opun_{\leq {#1}}}

\newcommand{\cat}{\mathrm{Cat}_{\infty}}
\newcommand{\Finstar}{\Fin_*}

\newcommand{\Opunleqk}{\Opunleq{k}}

\newcommand{\Spaces}{\mathcal{S}}

\newcommand{\Lk}{\mathrm{L}_k}
\newcommand{\Rk}{\mathrm{R}_k}

\tikzset{
    symbol/.style={
        draw=none,
        every to/.append style={
            edge node={node [sloped, allow upside down, auto=false]{$#1$}}}
    }
}

\def\AA{\mathbb A}
\def\EE{\mathbb E}

\def\cC{\mathcal C}\def\cD{\mathcal D}
\def\cF{\mathcal F}

\def\cO{\mathcal O}\def\cP{\mathcal P}
\def\cQ{\mathcal Q}\def\cR{\mathcal R}

\newcommand{\LKan}{\mathrm{LKan}}

\newcommand{\pt}{\mathrm{pt}}

\makeatletter
\newcommand{\colim@}[2]{
  \vtop{\m@th\ialign{##\cr
    \hfil$#1\operator@font colim$\hfil\cr
    \noalign{\nointerlineskip\kern-\ex@}\cr}}
}
\newcommand{\colim}{
  \mathop{\mathpalette\colim@{\scriptscriptstyle}}\nmlimits@
}

\newcommand{\Alg}{\mathrm{Alg}}
\newcommand{\CAlg}{\mathrm{CAlg}}
\newcommand{\fgt}{\mathrm{fgt}}
\newcommand{\onehalf}{\frac{1}{2}}

\usepackage{autonum}

\title{$\mathbb{E}_n$-algebras in $m$-categories}
\author{Yu Leon Liu}
\address{Department of Mathematics, Harvard University, 1 Oxford St, Cambridge, MA 02139}
\email{yuleonliu@math.harvard.edu}

\begin{document}
\begin{abstract}
    We prove a connectivity bound for maps of $\infty$-operads of the form $\AA_{k_1} \otimes \cdots \otimes \AA_{k_n} \to \EE_n$, and as a consequence, give an inductive way to construct $\EE_n$-algebras in $m$-categories. The result follows from a version of Eckmann-Hilton argument that takes into account both connectivity and arity of $\infty$-operads. 
    Along the way, we prove a technical Blakers-Massey type statement for algebras of coherent $\infty$-operads.
\end{abstract}

\maketitle
\tableofcontents

\section{Introduction}
In the pioneering work \cite{stasheff1963homotopy}, Stasheff constructs 
a sequence of convex polytopes $K_n$, called the \emph{Stasheff associahedra}, that encode the associativity of binary multiplications.
This construction can be understood within the framework of $\infty$-operads: let $\EE_1$ be the associative operad . 
There exists a converging filtration 
\begin{equation}\label{eq:Ak-filtration}
   \EE_0 =  \AA_1 \to \AA_2 \to \cdots \to \AA_k \to \cdots \to \AA_\infty = \EE_{1}
\end{equation}
where the $\infty$-operad $\AA_k$ parametrizes unital binary multiplications that are homotopy coherently associative up to $k$ objects.\footnote{See for example \cite[Example 4.15]{dubey2024unital}.} 
Furthermore, the obstruction theory associated to the $\AA_k$ filtration \eqref{eq:Ak-filtration} is well-understood:
extending an $\AA_{k-1}$-algebra structure to an $\AA_{k}$-algebra structure is equivalent to lifting certain maps from $\partial K_k$ to $K_k$.\footnote{The non-unital version of this statement is proven in \cite[Theorem 4.1.6.13]{HA}. The unital version that we stated follows from the main theorem of \cite{goppl}.}

For $k \geq 1$, the map $\AA_k \to \EE_1$ satisfies two properties:
\begin{enumerate}
    \item\label{enum-item:AAk-1} 
    The map $\AA_k \to \EE_1$ is multi-homwise $(k-3)$-connected (\cref{def:op-surj-on-obj-homwise}); that is, 
    for any $n$, the induced map on $n$-ary morphism spaces $\AA_k(n) \to \EE_1(n)$ is $(k-3)$-connected. 
    \item\label{enum-item:AAk-2} The $\infty$-operad $\AA_k$ is the ``$k$-arity restriction'' of $\EE_1$; that is, for $n \leq k$, the map $\AA_k(n) \to \EE_1(n)$ is an equivalence, and for $n > k$ the $n$-ary morphism space $\AA_k(n)$ is generated by $(\leq k)$-ary morphisms of $\AA_k$ (as well as  $\EE_1$) in a suitable sense.
\end{enumerate}

This gives an explicit way to construct $\EE_1$-algebras in $m$-categories. Recall that an $m$-category is an $\infty$-category whose hom spaces are $(m-1)$-truncated.\footnote{These $m$-categories are $(m,1)$-categories, that is, the $n$-morphisms are invertible for $n \geq 2$. It sufficient to consider $(m,1)$-categories as $\cO$-algebras in an $(m,n)$-category are equivalent to $\cO$-algebras in its underlying $(m,1)$-category.}
The connectivity property \eqref{enum-item:AAk-1} implies the following: for any $m \geq -1$ and any symmetric monoidal $m$-category $\cC$, $\EE_1$-algebras are $\cC$
    is equivalent to $\AA_{m+2}$-algebras in $\cC$. Explicitly, the restriction map 
    \begin{equation}
        \Alg_{\EE_1}(\cC) \to \Alg_{\AA_{m+2}}(\cC)
    \end{equation}
is an equivalence of $\infty$-categories.

In this paper we generalize this pattern to $\EE_n$-algebras. 
By the additivity theorem, proved by Dunn in \cite{Dunn} in the classical setting and by Lurie in the framework of $\infty$-categories \cite[Theorem 5.1.2.2]{HA}, we have an equivalence of $\infty$-operads
\begin{equation}
    \EE_n \simeq \EE_1^{\otimes n}.
\end{equation}
It follows that there is a filtration of $\EE_n$ by $\AA_{k_1} \otimes \cdots \otimes \AA_{k_n}$. 
We are interested in a variant of connectivity, called $d$-equivalence, for the map  $\AA_{k_1} \otimes \cdots \otimes \AA_{k_n} \to \EE_n$. For $d \geq -2$, we say that a map of $\infty$-operads $\cO_1 \to \cO_2$ is a \emph{$d$-equivalence} if it induces an equivalence on $d$-truncations (\cref{def:op-d-equivalence}). Equivalently, for any symmetric monoidal $(d+1)$-category $\cC$, the induced functor $\Alg_{\cO_2}(\cC) \to \Alg_{\cO_1}(\cC)$ is an equivalence of $\infty$-categories. Our main result is the following:
\begin{maintheorem}[{\cref{thm:AAk-equivalence}}]\label{thm:intro-main-thm}
    Fix $1 \leq k_1, \cdots, k_n \leq \infty$. Let $k_{\min}$ be $\min(k_1, k_2, \cdots, k_n)$ and $r$ be the number of occurences of $k_{\min}$ among $k_1, \cdots, k_n$. The map of $\infty$-operads 
    \begin{equation}
        \AA_{k_1} \otimes \cdots \otimes \AA_{k_n} \to \EE_n
    \end{equation}
    is a $(nk_{\min} - 2 - r)$-equivalence.
\end{maintheorem}
See \cref{thm:AAk-equivalence} for the full statement. 
The following is an immediate corollary:
\begin{corollary}\label{cor:intro-main-cor}
    Fix $1 \leq n \leq \infty$ and $m \geq -1$. Let $\cC$ be a symmetric monoidal $m$-category. The restriction functor  
    \begin{equation}
        \Alg_{\EE_n}(\cC) \to \Alg_{\AA_{k+1}^{\otimes(n-s)} \otimes \AA_{k}^{\otimes s}}(\cC)
    \end{equation}
    is an equivalence of $\infty$-categories, where $k = \lceil \frac{m+1}{n} \rceil + 1$ and $s$ is $(nk - m-1)$.
\end{corollary}   
Below is a table of $\AA_{k+1}^{\otimes(n-s)} \otimes \AA_{k}^{\otimes s}$ for small $n$ and $m$.
\begin{center}
    \begin{tabular}{ c| c| c| c|c}
     $m$-category &  $\EE_1$ & $\EE_2$ & $\EE_3$ &  $\EE_4$ \\ 
     \hline
     -1 & $\AA_1$ & $\AA_1 \otimes \AA_1$ & $\AA_1 \otimes \AA_1 \otimes \AA_1$ & $\AA_1 \otimes \AA_1 \otimes \AA_1 \otimes \AA_1$\\
     0 & $\AA_2$ & $\AA_2 \otimes \AA_1$ & $\AA_2 \otimes \AA_1 \otimes \AA_1$ & $\AA_2 \otimes \AA_1 \otimes \AA_1 \otimes \AA_1$\\
     1 & $\AA_3$ & $\AA_2 \otimes \AA_2$ & $\AA_2 \otimes \AA_2 \otimes \AA_1$ & $\AA_2 \otimes \AA_2 \otimes \AA_1 \otimes \AA_1$\\
     2 & $\AA_4$ & $\AA_3 \otimes \AA_2$ & $\AA_2 \otimes \AA_2 \otimes \AA_2$ & $\AA_2 \otimes \AA_2 \otimes \AA_2 \otimes \AA_1$\\  
     3 & $\AA_5$ & $\AA_3 \otimes \AA_3$ & $\AA_3 \otimes \AA_2 \otimes \AA_2$ & $\AA_2 \otimes \AA_2 \otimes \AA_2 \otimes \AA_2$\\
     4 & $\AA_6$ & $\AA_4 \otimes \AA_3$ & $\AA_3 \otimes \AA_3 \otimes \AA_2$ & $\AA_3 \otimes \AA_2 \otimes \AA_2 \otimes \AA_2$\\
    \end{tabular}
    \end{center} 
    Let us remind the reader that $\AA_1 = \EE_0$ is an idempotent in the $\infty$-category $\Opun$ of unital $\infty$-operads; that is, $\cO \otimes \AA_1 \simeq \cO$ for any unital $\infty$-operad $\cO$.\footnote{Recall that an $\infty$-operad $\cO$ is unital if for every color $X \in \cO$, the $0$-ary morphism space $\Mul_{\cO}(\varnothing, X)$ is contractible.}
\begin{remark}
    Heuristically, \cref{cor:intro-main-cor} states that the most optimal way to obtain an $\EE_n$-algebra in an $m$-category is by ``going up the diagonal'' in the $n$-dimensional lattice, where the point $(k_1, \cdots, k_n)$ corresponds to $\AA_{k_1} \otimes \cdots \otimes \AA_{k_n}$. Furthermore, going from $m$-categories to $(m+1)$-categories, we simply need to increase one of the minimum elements of $(k_1, \cdots , k_n)$ by one. 
\end{remark}
We refer the reader to  \cref{sec:En-in-m-cat} for  how \cref{cor:intro-main-cor} relates to the Eckmann-Hilton argument \cite{EH}
and classical notions of monoidal and braided monoidal structures.
 
\textbf{Generalized Eckmann-Hilton argument.}
\cref{thm:intro-main-thm} is the consequence of a generalized Eckmann-Hilton argument (EHA), which we now explain. 
First let us recall the $\infty$-categorical EHA given in \cite{EHA}. 
Recall that a unital $\infty$-operad is \emph{reduced} if the underlying $\infty$-category is contractible.
Let $\cP \to \cQ$ be a $d_1$-equivalence of reduced $\infty$-operads and $\cR$ be a multi-homwise $d_2$-connected reduced $\infty$-operad. The $\infty$-categorical EHA  \cite[Theorem 1.0.2]{EHA} implies that the induced map 
\begin{equation}
    \cP \otimes \cR \to \cQ \otimes \cR
\end{equation}
is a $(d_1 + d_2 + 2)$-equivalence.

Our generalized EHA not only takes in account of the connectivity of $\infty$-operads, but also their arity. 
Let us  review the theory of arity restricted $\infty$-operads, as developed in \cite{dubey2024unital}.\footnote{See also \cite{MR3450458, barkan2023,goppl}.} Given $k \geq 1$, a $k$-restricted $\infty$-operad is a variant of $\infty$-operad where we only consider $n$-ary morphism spaces for $n \leq k$. In \cite{dubey2024unital}, the author and a collaborator define unital $k$-restricted $\infty$-operads, which form an $\infty$-category $\Opunleqk$; furthermore, we show
that the natural restriction functor $(-)^k \colon \Opun \to \Opunleqk$ has fully faithful left and right adjoints, which we denote by $\Lk$ and $\Rk$ respectively. 
Let $\cO$ be a unital $\infty$-operad; it follows that we have a converging filtration
\begin{equation}
    \mathrm{L}_1\cO^1 \to \mathrm{L}_2\cO^2 \to \cdots \to \Lk\cO^k \to \cdots \to \cO.
\end{equation}
When $\cO = \EE_1$, this filtration is precisely the $\AA_k$ filtration \eqref{eq:Ak-filtration}: the $\infty$-operad $\AA_k$ is defined to be $\Lk(\EE_1)^k$ (see \cite[Example 4.15]{dubey2024unital}), and the natural map $\AA_k \to \EE_1$ is simply the unit map. This is the precise meaning of property \eqref{enum-item:AAk-2}. 

One version of the generalized Eckmann-Hilton argument  is the following:
\begin{maintheorem}[{\cref{cor:main-thm-coherent}}]\label{thm:intro-gen-EHA-coherent}
    Fix $k \geq 1$.
    Let $f \colon \cP \to \cQ$ be a multi-homwise $d_1$-connected map between reduced $\infty$-operads such that its $k$-restriction $f^k \colon \cP^k \to \cQ^k$ is an equivalence. Let $\cR$ be a coherent multi-homwise $d_2$-connected $\infty$-operad, then the induced map 
    \begin{equation}
        \cP \otimes \cR \to \cQ \otimes \cR
    \end{equation}
    is a $(d_1 + k (d_2 + 2))$-equivalence. 
\end{maintheorem}
Note that there is a assumption on  $\cR$, called coherence  (\cite[Definition 3.1.1.9]{HA}), which is a technical condition on reduced $\infty$-operads. This condition  ensures that the tensor product of the $\infty$-category  of $\cR$-algebras preserves pushouts in each variable ({\cite[Corollary 5.3.1.16]{HA}}). 
The proof of \cref{thm:intro-gen-EHA-coherent} utilizes \cref{prop:op-BM}, which is a Blakers-Massey type result for algebras of coherent $\infty$-operads. We refer the reader to \cref{sec:BM-for-alg} for a review on the (dual) Blakers-Massey theorem as well as our generalization.

The last difficulty comes from the fact that while the $\EE_n$ operads are coherent ({\cite[Theorem 5.1.1.1]{HA}}), the $\AA_k$ operads for $1 \leq k \leq \infty$ are not (\cref{rem:Ak-not-coherent}). Therefore we cannot direct apply \cref{thm:intro-gen-EHA-coherent}. 
To remedy this, we introduce the notion of $k$-wise $d$-connected $\infty$-operads  and prove the following properties:
\begin{enumerate}
    \item A coherent multi-homwise $d$-connected $\infty$-operad is $k$-wise $k(d+2)-2$-connected.
    \item Given a map of reduced $\infty$-operads $\cO_1 \to \cO_2$, if $f$ is a $d$-equivalence and $\cO_2$ is $k$-wise $d$-connected, then $\cO_1$ is also $k$-wise $d$-connected.
\end{enumerate}
This allow us to 
use the connectivity bound on $\AA_n \to \EE_1$ to show that $\AA_n$ is $k$-wise $\min(n-3, k-2)$-connected (\cref{cor:AAn-k-wise-d-conn}).
Additionally, the generalized Eckmann-Hilton argument \cref{thm:intro-gen-EHA-coherent} also holds in the setting of $k$-wise $d$-connected $\infty$-operads:
\begin{maintheorem}[{\cref{thm:main-thm}}]\label{thm:intro-gen-EHA-k-wise}
    Let $f \colon \cP \to \cQ$ be a multi-homwise $d_1$-connected map between reduced $\infty$-operads such that the $k$-restriction $f^k \colon \cP^k \to \cQ^k$ is an equivalence. Let $\cR$ be a reduced $k$-wise $d_2$-connected $\infty$-operad, then the induced map 
    \begin{equation}
        \cP \otimes \cR \to \cQ \otimes \cR
    \end{equation}
    is a $(d_1 + d_2 + 2)$-equivalence.
\end{maintheorem}
Finally \cref{thm:intro-main-thm} follows from \cref{thm:intro-gen-EHA-k-wise} via an inductive argument.

\textbf{Outline:} In \cref{sec:conn-trunn} we review the various notions of connectedness and truncatedness for objects in $\infty$-categories 
as well as for $\infty$-categories and $\infty$-operads themselves.
In \cref{sec:BM-for-alg} we prove  as \cref{prop:op-BM} the Blakers-Massey type statement for algebras of coherent $\infty$-operads.
After reviewing the notion of $k$-restricted $\infty$-operads, 
 we define $k$-wise $n$-connected $\infty$-operads and prove  \cref{thm:intro-gen-EHA-k-wise} in \cref{sec:gen-EHA}. Finally, in \cref{sec:En-in-m-cat} we prove \cref{thm:intro-main-thm} and spell out its consequences for $m$-categories with $-1 \leq m \leq 2$.

\textbf{Acknowledgements:}
The author would like to thank Aaron Mazel-Gee, Catharina Stroppel, and Paul Wedrich for previous collaboration, from which this work grew out of. The author would especially like to thank David Reutter for the previous collaboration and for pointing out the crucial role of coherent $\infty$-operads. The author is also grateful to Shaul Barkan, Tomer Schlank, and Andy Senger for helpful discussions. Lastly, the author gracially thank Amartya Dubey for prior collaboration that this work depends on, and Cameron Krulewski for helpful comments on an early version of the draft.

The author gratefully acknowledges the financial support provided by the Simons Collaboration on Global Categorical Symmetries.

\section{Connectedness and truncatedness}\label{sec:conn-trunn}
\subsection{For objects in an $\infty$-categories}\label{subsec:conn-trunn-for-obj}
Let $\Spaces$ be the $\infty$-category of spaces.
\begin{definition}\label{def:conn-trun-for-spaces}
    For $d \geq 0$, a space $X$ is called \emph{$d$-truncated} if $\pi_i(X, x) = 0$ for all $i > d$ and $x \in X$. In addition, a space $X$ is called \emph{$d$-connected} if $X$ is connected and $\pi_i(X, x) = 0$ for all $i \leq d$ and $x \in X$.
    For $d = -1$, we declare that a space $X$ is \emph{$(-1)$-truncated} if it is nonempty, and $X$ is \emph{$(-1)$-connected} if it is either empty or contractible. 
    For $d = -2$, we declare that a space $X$ is \emph{$(-2)$-truncated} if it is contractible, and $X$ is always \emph{$(-2)$-connected}.

    For $d \geq -2$, we declare that a map of spaces is \emph{$d$-truncated} (\emph{$d$-connected}) if its fibers are all $d$-truncated (resp. $d$-connected).
\end{definition}
By \cite[Example 5.2.8.16]{HTT}, the classes of ($d$-connected, $d$-truncated) maps form a factorization system on $\Spaces$.\footnote{We refer the reader to \cite[\S 5.2.8]{HTT} for an introduction to factorization systems.}
\begin{warning}\label{warning:d-conn-d-1-conn}
    Readers be warned that $d$-connected objects and morphisms would be called $(d+1)$-connected in many parts of the literature, such  as in 
    \cite{Goo}, \cite[Definition 6.5.1.10]{HTT}, and \cite{cubical}.
\end{warning}

The notions of truncatedness and connectedness naturally extend to objects and morphisms in general $\infty$-categories:
\begin{definition}[{\cite[Definition 5.5.6.1]{HTT}}]\label{def:d-truncated-for-general-cat}
    For $d \geq -2$, 
    an object $X$ in $\cC$ is \emph{$d$-truncated} if for every $Z \in \cC$, the morphism space $\Hom_{\cC}(Z, X)$ is $d$-truncated.
    Similarly, a map $f \colon X \to Y$ in an $\infty$-category $\cC$ is called \emph{$d$-truncated} if for every $Z \in \cC$, the induced map 
    \begin{equation}
        \Hom_{\cC}(Z, X) \to \Hom_{\cC}(Z, Y)
    \end{equation}
    is a $d$-truncated map of spaces. 
\end{definition}
If $\cC$ has a terminal object $*_\cC$, then an object $X$ in $\cC$ is $d$-truncated if and only if the map $X \to *_{\cC}$ is $d$-truncated.
\begin{definition}\label{def:d-connected-for-general-cat}
    For $d \geq -2$, a map $f \colon A \to B$ in an $\infty$-category $\cC$ is called \emph{$d$-connected} if it has the left lifting property against $d$-truncated morphisms in $\cC$, that is, for every $d$-truncated morphism $g \colon X \to Y$ and  commutative diagram
    \begin{equation}
        \begin{tikzcd}
        A \ar[r] \ar[d, "f"] & X \ar[d, "g"] \\ 
        B \ar[r] \ar[ru, dashed] & Y,
        \end{tikzcd}
    \end{equation}
    there exists a contractible space of lifts.
\end{definition}
\begin{observation}\label{obs:lim-colim-preserves}
   Let $\cC$ be an $\infty$-category. Colimits in the arrow category $\Fun(\Delta^1, \cC)$ preserve $d$-connected morphisms, while limits in the arrow category $\Fun(\Delta^1, \cC)$  preserves $d$-truncated morphisms. 
\end{observation}
\begin{observation}[{\cite[Proposition 4.6]{gep17}}]\label{obs:presentable-d-fact}
    Let $\cC$ be a presentable $\infty$-category (\cite[Definition 5.5.0.1]{HTT}). For any $d \geq -2$, the classes of ($d$-connected, $d$-truncated) morphisms in $\cC$ define a factorization system. 
\end{observation}
Let $\cC$ be a presentable $\infty$-category. For $d \geq -2$, we denote by $\cC_{\leq d}$ the full subcategory of $d$-truncated objects, and $\tau_{\leq d} \colon \cC \to \cC_{\leq d}$ the truncation functor induced by the factorization system.

Recall that an $\infty$-topos is a left exact localization of $\Spaces^{K}$ for some small $\infty$-category $\cC$ (\cite[Definition 6.1.0.2]{HTT}).
\begin{definition}\label{def:d-topos}
    Fix $m \geq -1$. An $\infty$-category
    is a \emph{$m$-topos} if it is the full subcategory on $(m-1)$-truncated objects in an $\infty$-topos.
\end{definition}
\begin{remark}
    By \cite[Theorem 6.4.1.5]{HTT}, the definition above is equivalent to the original notion defined in \cite[Definition 6.4.1.1]{HTT}.
\end{remark}

\begin{definition}[{\cite[\S 4.3]{EHA}}]\label{def:d-1/2-eq}
   Let $\cC$ a presentable $\infty$-category. For $d \geq -2$, a morphism $f \colon X \to Y$ in $\cC$ is called \emph{$(d-\onehalf)$-connected} if the induced map $\tau_{\leq d} f \colon \tau_{\leq d} X \to \tau_{\leq d} Y$ is an equivalence.
\end{definition}
By \cite[Lemma. 4.3.2]{EHA}, $d$-connected morphisms in a presentable $\infty$-category are $(d-\onehalf)$-connected. On the other hand, by \cite[Lemma 4.3.4]{EHA}, $(d+\onehalf)$-morphisms in an $m$-topos are $d$-connected.

\subsection{For $\infty$-categories and $\infty$-operads}\label{subsec:conn-trunn-inf-cat}
In this subsection we extend the concepts of connectedness and truncatedness to $\infty$-categories and $\infty$-operads. 
For a comprehensive study of these notions in $(\infty, n)$-categories, we direct the reader to  \cite[\S 5]{braiding}.
\begin{definition}\label{def:cat-surj-on-obj-homwise}
    A functor of $\infty$-categories $F \colon \cC \to \cD$ is \emph{surjective on objects} if every object $d \in \cD$ is of the form $F(c)$ for some $c \in \cC$.
    For $d \geq -2$, a functor of $\infty$-categories $F \colon \cC \to \cD$ is \emph{homwise $d$-truncated} (resp. \emph{homwise $d$-connected}) if for every pair of objects $c_1, c_2 \in \cC$, the induced map 
    \begin{equation}
        \Hom_{\cC}(c_1, c_2) \to \Hom_{\cD}(f(c_1), f(c_2))
    \end{equation}
    is $d$-truncated (resp. $d$-connected).

    For $d \geq -1$, an $\infty$-category $\cC$ is called a \emph{$d$-category} if 
    the functor $\cC \to \pt$ is homwise $(d-1)$-truncated; equivalently, 
    for every $c_1, c_2 \in \cC$, the space $\Hom_{\cC}(c_1, c_2)$ is $(d-1)$-truncated.
\end{definition}
\begin{example}\label{ex:d-1-truncated-spaces}
    For $d \geq -1$, the full subcategory $\Spaces_{\leq (d-1)}$ of $(d-1)$-truncated spaces is a $d$-category. In particular, the category of sets $\Spaces_{\leq 0}$ is a $1$-category. More generally, $d$-topoi are $d$-categories.
\end{example}
By \cite[Theorem 5.3.7]{braiding}, the classes of (surjective on objects and homwise $d$-connected, homwise $d$-truncated) morphisms form a factorization system on the $\infty$-categories $\cat$ of 
large
$\infty$-categories.
Let $\Cat_d$ denote the full $\infty$-category of $d$-categories. By the factorization system we have a localization functor $h_d \colon \Cat \to \Cat_d$, left adjoint to the fully faithful inclusion $\Cat_d \hookrightarrow \Cat$.

\begin{warning}\label{warn:2-diff-fact}
    The notion of $d$-truncated morphisms in the $\infty$-category $\cat$ is different yet related to the notion of homwise $d$-truncated 
    functors. We refer the reader to  \cite[Warning 5.3.9]{braiding} for a detailed discussion.
\end{warning}
\begin{definition}\label{def:cat-d-equivalence}
    For $d \geq -2$, a functor of $\infty$-categories $F \colon \cC \to \cD$ is called a \emph{$d$-equivalence} if the induced functor $h_d F \colon h_d \cC \to h_d \cD$ is an equivalence.
 \end{definition}
Note that a surjective on objects and homwise $d$-connected functor betwen $\infty$-categories is a $d$-equivalence.
Next we turn to $\infty$-operads.
\begin{definition}\label{def:op-surj-on-obj-homwise}
    A functor of $\infty$-operads is \emph{surjective-on-colors} if the induced map on underlying $\infty$-categories is surjective on objects.

    For $d \geq -2$, a map of $\infty$-operads $f \colon \cO \to \cO'$ is \emph{multi-homwise $d$-truncated} (resp. \emph{multi-homwise $d$-connected}) if for every list of objects $X_1, \cdots, X_n, Y$ in $\cO$, the induced map 
    \begin{equation}
        \Mul_{\cO}(X_1, \cdots, X_n; Y) \to \Mul_{\cO'}(f(X_1), \cdots, f(X_n); f(Y))
    \end{equation}
    is $d$-truncated (resp. $d$-connected).
    
    An $\infty$-operad $\cO$ is called a \emph{$d$-operad} if 
    the map $\cO \to \EE_\infty$ is homwise $(d-1)$-truncated. Equivalently, 
    for every list of objects $X_1, \cdots, X_n, Y$ in $\cO$, the space $\Mul_{\cO}(X_1, \cdots, X_n; Y)$ is $(d-1)$-truncated. 
\end{definition}
\begin{remark}
    By \cite[Lemma 7.5.2]{braiding}, a map of $\infty$-operads $f \colon \cO \to \cO'$ is surjective on colors (resp. multi-homwise $d$-truncated, multi-homwise $d$-connected) if and only if the functor on total space $f^{\otimes} \colon \cO^{\otimes} \to (\cO')^{\otimes}$ is surjective on objects (resp. homwise $d$-truncated, homwise $d$-connected).
\end{remark}

\begin{example}\label{ex:reduced-surjective-on-colors}
    Recall that an $\infty$-operad $\cO$ is reduced if it is unital and the underlying $\infty$-category is equivalent to $\pt$. Maps between reduced $\infty$-operads are automatically surjective on colors.
\end{example}
\begin{example}\label{ex:connected-level}
    Given $0 \leq n < m \leq \infty$, the canonical map $\EE_n \to \EE_m$ is multi-homwise $(n-2)$-connected. 
    Given $1 \leq n < m \leq \infty$, the canonical map 
    $\AA_n \to \AA_m$ is multi-homwise $(n-3)$-connected. This follows from the explicit description of the multi-ary morphism spaces in \cite[Example 4.1.5]{dubey2024unital}.
\end{example}
\begin{observation}\label{obs:op-fact-sys}
    By \cite[Proposition 7.5.3]{braiding}, the classes of (surjective on colors and multi-homwise $d$-connected, multi-homwise $d$-truncated)
    morphisms form a factorization system on the $\infty$-category $\Op$ of large $\infty$-operads. Let $\Op_d$ be the full subcategory of $\Op$ consisting of $d$-operads. We have a localization functor $h_d \colon \Op \to \Op_d$, left adjoint to the fully faithful inclusion $\Op_d \hookrightarrow \Op$.\footnote{Strict models of $d$-categories and $d$-operads are studied in \cite{SY19cat}.}  
\end{observation}

\begin{observation}\label{obs:symmetric-monoidal-d-op}
    A symmetric monoidal $\infty$-category is a $d$-operad if and only if the underlying $\infty$-category is a $d$-category.
\end{observation}
Lastly, we have the notion of $d$-equivalence for maps between $\infty$-operads.
\begin{definition}\label{def:op-d-equivalence}
    For $d \geq -2$, a map of $\infty$-operads $f \colon \cO \to \cO'$ is called a \emph{$d$-equivalence} if the induced functor $h_d f \colon h_d \cO \to h_d \cO'$ is an equivalence.
\end{definition}
Note that a surjective on colors and multi-homwise $d$-connected map of $\infty$-operads is a $d$-equivalence.
\begin{observation}\label{obs:d-equivalence-d-cat}
    Let $f \colon \cO \to \cO'$ be a  map of $\infty$-operads  and $\cC$ be a symmetric monoidal $\infty$-category. If $f$ is a $d$-equivalence and $\cC$ is a $(d+1)$-category for some $d \geq -2 $, by \cref{obs:symmetric-monoidal-d-op} the induced map 
    \begin{equation}
        \Alg_{\cO'}(\cC) \to \Alg_{\cO}(\cC)
    \end{equation}
    is an equivalence of $\infty$-categories.
\end{observation}

\section{Blakers-Massey theorem for algebras of coherent $\infty$-operads}
\label{sec:BM-for-alg}
In this section we prove  \cref{prop:op-BM}, which is a Blakers-Massey type result for algebras of coherent $\infty$-operads. 
After introducing basic concepts in cubical homotopy theory in  \cref{subsec:cubes}, we state the \cref{prop:op-BM} and give some applications in \cref{subsec:BM-for-alg}. Finally, we prove \cref{prop:op-BM} in \cref{subsec:proof-of-op-BM}.
\subsection{Recollection on cubes}\label{subsec:cubes}

Let $S$ be a finite set and $\cP(S)$ be the power set of $S$, which we view as a category with morphisms being inclusions. For $k \geq 0$, let $\cP_{\leq k}(S)$ (resp. $\cP_{\geq k}(S)$) be the full subcategory of $\cP(S)$ consisting of $I \subset S$ with $|I| \leq k$ (resp. $|I| \geq k$). 
For $n \geq 1$ we denote by $\underline{n}$ the finite set $\{1, \cdots, n\}$. We take $\underline{0}$ to be the empty set $\varnothing$.
\begin{definition}\label{def:S-cube}
   Let $S$ be a finite set and $\cC$ be an $\infty$-category. An \emph{$S$-cube} in $\cC$ is a functor $\cF \colon \cP(S) \to \cC$. 
   Given an $S$-cube  $\cF \colon \cP(S) \to \cC$, we denote by $\cF|_{\leq k}$ (resp. $\cF|_{\geq k}$) for the restriction of $\cF$ on $\cP_{\leq k}(S)$ (resp. $\cP_{\geq k}(S)$). Similarly, a \emph{punctured $S$-cube} in $\cC$ is a functor $\cF \colon \cP_{\leq |S|-1}(S) \to \cC$.

   For $d \geq -2$, we will say that an $S$-cube $\cF \colon \cP(S) \to \cC$ is \emph{$d$-coCartesian} if the colimit $\colim(\cF|_{\leq |S|-1})$ exists and the canonical map 
   \begin{equation}
    \colim(\cF|_{\leq |S|-1}) \to \cF(S)
   \end{equation}
   is $d$-connected. When $d = \infty$ we simply say that it is coCartesian.

   We will say that an $S$-cube $\cF \colon \cP(S) \to \cC$ is \emph{strongly Cartesian} if $\cF$ is the right Kan extension of $\cF_{\geq |S| - 1}$ along the inclusion $\cP_{\geq |S| -1}(S)\hookrightarrow \cP(S)$.
\end{definition}
\begin{observation}\label{obs:d-coCart-colimit}
    Let $\cC$ be an $\infty$-category and $S$ be a finite set. By \cref{obs:lim-colim-preserves}, the $d$-coCartesian squares are closed under colimits in $\Fun(\cP(S), \cC)$.
\end{observation}
\begin{example}\label{ex:0-cube}
    Let $\cC$ be an $\infty$-category.
    A $0$-cube in $\cC$ is an object of $\cC$. A $0$-cube is coCartesian if the object is an initial object, and it is always strongly Cartesian.
\end{example}

\begin{example}\label{ex:1-cube}
    A $1$-cube in $\cC$ is a morphism in $\cC$. For $d \geq -2$, 
    a cube is $d$-coCartesian if the morphism is $d$-connected,  and it is always strongly Cartesian.
\end{example}
We will often denote a $1$-cube by the corresponding morphism.
\begin{example}\label{ex:2-cube}
    A $2$-cube in $\cC$ is a commutative square 
    \begin{equation}
        \begin{tikzcd}
            A \ar[r] \ar[d] & B \ar[d]\\ 
            C \ar[r] & D
        \end{tikzcd}
    \end{equation}
    in $\cC$. For $d \geq -2$, such a square is $d$-coCartesian if the map $B \sqcup_A C \to D$ is $d$-connected. Such a square is strongly Cartesian if the map $A \to B \times_{C} D$ is an equivalence, i.e., the square is a pullback square.
\end{example}

\begin{lemma}\label{lem:d-coCart-section-lemma}
Let $\cC$ be an $m$-topos for some $-1 \leq m \leq \infty$. Suppose we have the following diagram 
\begin{equation}\label{eq:d-coCart-section-lemma-1}
    \begin{tikzcd}
        A \ar[r] \ar[d] & B \ar[d] \\ 
        C \ar[r] \ar[d] & D \ar[d] \\ 
        E \ar[r]        & F.
    \end{tikzcd}
\end{equation}
If the top square is $d$-coCartesian and the total square is $(d+1)$-coCartesian, then the bottom square is $(d+1)$-coCartesian.
\end{lemma}
\begin{proof}
    This follows from the same argument as in \cite[\S5.8.14(b)]{cubical}.
\end{proof}

Let us give a useful way to construct new cubes:
\begin{example}\label{ex:tensor-cubes}
    Let $\cC$ be a symmetric monoidal $\infty$-category with tensor product $\otimes$. For $i = 1, 2$, let $S_i$ be a finite set and $\cF_i$ be a $S_i$-cube in $\cC$. Let $S = S_1 \sqcup S_2$, we can construct a $S$-cube $\cF_1 \square \cF_2$ as follows:
    \begin{enumerate}
        \item For $I \subset S$, we define $\cF_1\square \cF_2(I)$ to be $\cF_1(I \cap S_1) \otimes \cF_2(I \cap S_2)$.
        \item For an inclusion of subsets $I \subset J$, the map $\cF_1\square \cF_2(I) \to \cF_1\square \cF_2(J)$ is given by $\cF(I \cap S_1 \to J \cap S_1) \otimes \cF(I \cap S_2 \to J \cap S_2)$.
    \end{enumerate} 
\end{example}
Let $\cC$ be a $\infty$-category with finite products. By \cite[\S 2.4.1]{HA}, there is a essentially unique symmetric monoidal structure on $\cC$, called the \emph{Cartesian} symmetric monoidal structure, where the unit is the terminal object and the tensor product is the Cartesian product. 

\begin{lemma}\label{lem:square-strongly-coCart}
    Let $\cC$ be a symmetric monoidal $\infty$-category. For $i = 1, 2$, let $S_i$ be a finite set and $\cF_i$ be a strongly Cartesian $S_i$-cube in $\cC$.
    If the symmetric monoidal structure on $\cC$ is Cartesian, then $\cF_1 \square \cF_2$ is also strongly Cartesian.
\end{lemma}
\begin{proof}
    Note that 
    the commutative diagram 
    \begin{equation}
        \begin{tikzcd}
            \pt \ar[r, "S_1"] \ar[d, "S_2"] & \cP_{\geq |S_1| - 1}(S_1) \ar[d, "- \sqcup S_2"] \\ 
            \cP_{\geq |S_2| - 1}(S_2) \ar[r, "S_1 \sqcup - "] & \cP_{\geq |S|-1}(S) 
        \end{tikzcd}
    \end{equation}
    is a pushout of $\infty$-categories.
    For any $I \subset S$, we have a sequence of equivalences:
    \begin{equation}
        \begin{aligned}
            \cF(I) & = \cF_1(I \cap S_1) \times \cF_2(I \cap S_2) \\ 
                &\simeq \lim_{J_1 \in (\cP_{\geq |S_1| - 1}(S_1))_{I \cap S_1/}} \cF_1(J_1) \times \lim_{J_2 \in (\cP_{\geq |S_2| - 1}(S_2))_{I \cap S_2/}} \cF_2(J_2) \\ 
                &\simeq \left(\lim_{J_1 \in (\cP_{\geq |S_1| - 1}(S_1))_{I \cap S_1/}} \cF_1(J_1) \times \cF_2(S_2)\right) \times_{(\cF_1(S_1) \times \cF_2(S_2))} \left(\lim_{J_2 \in (\cP_{\geq |S_2| - 1}(S_2))_{I \cap S_2/}} \cF_1(S_1) \times \cF_2(J_2)\right) \\ 
                &\simeq \left(\lim_{J_1 \in (\cP_{\geq |S_1| - 1}(S_1))_{I \cap S_1/}} \cF(J_1 \sqcup S_2)\right) \times_{\cF(S)} \left(\lim_{J_2 \in (\cP_{\geq |S_2| - 1}(S_2))_{I \cap S_2/}} \cF(S_1 \sqcup J_2)\right) \\ 
                &\simeq \lim_{J \in (\cP_{\geq |S|-1}(S))_{I/}} \cF(J)
        \end{aligned}
    \end{equation}
    It follows that $\cF$ is indeed the right Kan extension of $\cF|_{\geq |S|-1}$.
\end{proof}

\begin{definition}\label{def:boxdot}
    Given a symmetric monoidal $\infty$-category $\cC$ with pushouts and morphisms $f_i \colon X_i \to Y_i$ in $\cC$ for $i = 1,2$. We denote by $f_1 \boxdot f_2$ the map 
    \begin{equation}
        \colim(f_1 \square f_2|_{\leq 1}) =  X_1 \otimes Y_2 \sqcup_{X_1 \otimes X_2} Y_1 \otimes X_2  \to Y_1 \otimes Y_2.
    \end{equation}
\end{definition}
By definition, the square $f_1 \square f_2$ is $d$-coCartesian if and only if the map $f_1 \boxdot f_2$ is $d$-connected.

\begin{notation}\label{nota:iterative-square}
    Given $\cC$ a symmetric monoidal $\infty$-category and $S$ a finite set. Suppose for each $i \in \underline{n}$ we have a morphism $f_i \colon X_i \to Y_i$ in $\cC$. We will collectively denote the maps as $f_{\underline{n}}$. We denote by $\square_{f_{\underline{n}}}$ the $n$-cube $((f_1 \square f_2) \square \cdots) \square f_n$. In the case that $\cC$ is unital and the morphisms $f_i$ are unit maps $1_{\cC} \to Y_i$, we will denote the $n$-cube $\square_{f_{\underline{n}}}$ by $\square_{Y_{\underline{n}}}$. 
\end{notation}

The following is an iterative way to compute punctured colimits of cubes:
\begin{lemma}[{\cite[Lemma 5.7.6]{cubical}}]\label{lem:iter-punctured-colim}
    Let $\cC$ be an $\infty$-category with pushouts, $S$ be a finite set, and $s$ be an element of $S$. 
    Given a punctured $S$-cube $\cF \colon \cP_{\leq |S|-1}(S) \to \cC$, the following commutative diagram
    \begin{equation}
        \begin{tikzcd}
            \colim_{I \in \cP_{\leq |S|-2}(S - s)}\cF(I) \ar[r] \ar[d] & 
            \colim_{I \in \cP_{\leq |S|-2}(S - s)}\cF(I \sqcup \{s\}) \ar[d]\\ 
            \cF(S-s) \ar[r] & 
            \colim_{I \in \cP_{\leq |S|-1}}\cF(I) =  \colim(\cF|_{\leq |S|-1}) 
        \end{tikzcd}
       \end{equation}
       is a pushout square in $\cC$.
\end{lemma}

We end this subsection with the following useful proposition:
\begin{proposition}\label{prop:iter-colim}
    Let $\cC$ be a symmetric monoidal $\infty$-category with small colimits and $n \geq 2$. Suppose we have $f_i \colon X_i \to Y_i$ in $\cC$ for $1 \leq i \leq n$. If the tensor product $\otimes$ preserves pushouts in each variable, i.e., the functor $- \otimes X \colon \cC \to \cC$ preserves pushouts for every objects $X \in \cC$, then the map 
    \begin{equation}\label{eq:iter-colim-eq-1}
        \colim(\square_{f_{\underline{n}}}|_{\leq n-1}) \to  \square_{f_{\underline{n}}}(\underline{n}) = \bigotimes_{i = 1}^n Y_i
    \end{equation}
    can be identified with 
    \begin{equation}
        ((f_1 \boxdot f_2) \boxdot \cdots) \boxdot f_n.
    \end{equation}
\end{proposition}
\begin{proof}
   The statement is vacuously true for the base case $n = 2$. Now we induct on $n$. By \cref{lem:iter-punctured-colim}, the commutative diagram 
   \begin{equation}
    \begin{tikzcd}
        \colim_{I \in \cP_{\leq n-2}(\underline{n-1})}\square_{f_{\underline{n}}}(I) \ar[r] \ar[d] & 
        \colim_{I \in \cP_{\leq n-2}(\underline{n-1})}\square_{f_{\underline{n}}}(I \sqcup \{n\}) \ar[d]\\ 
        \square_{f_{\underline{n}}}(\underline{n-1}) \ar[r] & 
        \colim(\square_{f_{\underline{n}}}|_{\leq n-1}) 
    \end{tikzcd}
   \end{equation}
   is a pushout square. This unpacks to 
   \begin{equation}
    \begin{tikzcd}
         \colim_{I \in \cP_{\leq n-2}(\underline{n})}  (\square_{f_{\underline{n-1}}} (I) \otimes X_n) \ar[r] \ar[d] & 
        \colim_{I \in \cP_{\leq n-2}(\underline{n})}  (\square_{f_{\underline{n-1}}} (I) \otimes Y_n)  \ar[d] \\ 
        \bigotimes_{i = 1}^{n-1} Y_i \otimes X_n  = \square_{f_{\underline{n-1}}}(\underline{n-1}) \otimes X_n\ar[r] & \colim(\square_{f_{\underline{n}}}|_{\leq n-1}),
    \end{tikzcd}
   \end{equation}
   where $\square_{f_{\underline{n-1}}}$ is $((f_1 \square f_2) \square \cdots) \square f_{n-1}$.
   By \cref{lem:iter-punctured-colim}, the tensor product $\otimes$ preserves colimits of punctured cubes in each variable. It follows that 
   \begin{equation}
    \colim_{I \in \cP_{\leq n-2}(\underline{n})}  (\square_{f_{\underline{n-1}}} (I) \otimes X) \simeq (\colim_{I \in \cP_{\leq n-2}(\underline{n})}  \square_{f_{\underline{n-1}}} (I)) \otimes X = \colim(\square_{f_{\underline{n-1}}}|_{\leq n-2}) \otimes X
   \end{equation}
  for any $X \in \cC$.
  Hence the map in \eqref{eq:iter-colim-eq-1} can be identified with 
  \begin{equation}
   \left(\colim(\square_{f_{\underline{n-1}}}|_{\leq n-2}) \to \square_{f_{\underline{n-1}}}(\underline{n-1})\right) \boxdot  (X_n \to Y_n).
  \end{equation}
 Now the result follows from the fact that $\colim(\square_{f_{\underline{n-1}}}|_{\leq n-2}) \to \square_{f_{\underline{n-1}}}(\underline{n-1})$ is equivalent to $((f_1 \boxdot f_2) \boxdot \cdots) \boxdot f_{n-1}$ by induction.
\end{proof}

\subsection{Blakers-Massey for algebras}\label{subsec:BM-for-alg}
Let us recall the statement of Goodwillie's cubical formulation of the dual Blakers-Massey theorem \cite[Theorem 2.4]{Goo}:
\begin{proposition}\label{prop:BM}
    Let $\cC$ be an $m$-topos for some $-1 \leq m \leq \infty$.
    Given a finite set $S$  and  a strongly Cartesian $S$-cube $\cF \colon \cP(S) \to \cC$. For each $i \in S$, suppose that the map $\cF(S - i) \to \cF(S)$ is $d_i$-connected.
    Then $\cF$ is $(\sum_{i \in S}(d_i+2) - 2)$-coCartesian. 
\end{proposition}
\begin{proof}
    The case of $\cC = \Spaces$ is proven in  \cite[Theorem 2.4]{Goo}. However, it can be extended to general $\infty$-topoi: the case of $n = 2$ is proven in \cite[Corollary 3.5.2]{GenBM}, while the general case follows from the same argument as \cite[Theorem 6.2.2]{cubical}. 
    The case of $m$-topoi follows from the fact that any $m$-topos is the full subcategory of $(m-1)$-truncated objects in some $\infty$-topos, and the truncation functor preserves connectedness.
\end{proof}

We are interested in proving a version of \cref{prop:BM} for $\cO$-algebras  and for cubes of the form $\square_{f_{S}}$ defined in \cref{nota:iterative-square}. 
There is a property on the $\infty$-operad, called \emph{coherent}, which we now discuss.

In \cite[Definition 3.1.1.9]{HA}, a technical condition on reduced $\infty$-operads called \emph{coherence} is introduced. It is a technical condition that guarantees a nice theory of $\cO$-module categories. 
In this paper, we only need to the following two results from \cite{HA} about coherent $\infty$-operads. The first is that many of the $\infty$-operads we are interested in are coherent:
\begin{proposition}[{\cite[Theorem 5.1.1.1]{HA}}]\label{prop:Ek-coherent}
    For $k \geq 0$, the little cubes $\infty$-operad $\EE_k$ is coherent.
\end{proposition}
\begin{remark}\label{rem:Ak-not-coherent}
    For $1 < k < \infty$, the $\infty$-operad $\AA_k$ is \emph{not} coherent. 
\end{remark}
Let $\cC$ be a symmetric monoidal $\infty$-category and $\cO$ be an $\infty$-operad. By \cite[Proposition 3.2.4.3]{HA}, there exists a symmetric monoidal structure on $\Alg_{\cO}(\cC)$, such that for any color $X \in \cO$, the evaluation functor $\mathrm{ev}_{X} \colon \Alg_{\cO}(\cC) \to \cC$ is symmetric monoidal. 
Furthermore, by \cite[Lemma 3.2.5]{EHA}, the symmetric monoidal structure on $\Alg_{\cO}(\cC)$ is Cartesian if the one on $\cC$ is such.
From now on we will always equip $\Alg_{\cO}(\cC)$ with its symmetric monoidal structure.

The second result on coherent $\infty$-operad allows us to apply \cref{prop:iter-colim} to the $\infty$-category of $\cO$-algebras.

\begin{proposition}[{\cite[Corollary 5.3.1.16]{HA}}]\label{prop:O-alg-preserves-pushout} 
   Let $\cO$ be a coherent $\infty$-operad and $\cC$ be a presentably symmetric monoidal $\infty$-category, that is, a presentable symmetric monoidal $\infty$-category whose tensor product preserves small colimits in each variables. 
   Let $K$ be a small simplicial set that is weakly contractible. 
   Then the tensor product on $\Alg_{\cO}(\cC)$ preserves $K$-indexed colimits in each variable. In particular it preserves pushouts in each variable.
\end{proposition}
\begin{remark}
    The symmetric monoidal structure on $\Alg_{\cO}(\cC)$ generally does not preserve all small colimits. In particular, it generally does not preserve coproducts. 
\end{remark}

We are now ready to state our Blakers-Massey type statement:
\begin{proposition}\label{prop:op-BM}
    Fix $d \geq -2$ and $-1 \leq m \leq \infty$.
    Let $\cO$ be a coherent multi-homwise $d$-connected $\infty$-operad and $\cC$ be an $m$-topos equipped with the Cartesian symmetric monoidal structure.
    Let $S$ be a finite set and $f_S$ be a collection of $d_i$-connected morphisms $\{f_i \colon X_i \to Y_i\}_{i \in S}$ in $\Alg_{\cO}(\cC)$. 
    Let $\square_{f_S}$ be the $S$-cube $\cP(S) \to \Alg_{\cO}(\cC)$ defined in \cref{nota:iterative-square}. Then  $\square_{f_S}$ is $$\sum_{i \in S} (d_i + 2) + (n-1)(d+2) - 2$$-coCartesian. 
    That is, 
    the canonical map 
    \begin{equation}\label{eq:punctured-morphism}
     \colim(\square_{f_S}|_{\leq n-1}) \to \square_{f_S}(S) =  \prod_{i \in S} Y_i
    \end{equation}
    is $
    \left(\sum_{i \in S} (d_i + 2) + (n-1)(d+2) - 2\right)$-connected.
 \end{proposition}
 \begin{remark}
     Note that the colimit in \eqref{eq:punctured-morphism} is taken in the $\infty$-category $\Alg_{\cO}(\cC)$ rather than in $\cC$.
 \end{remark}
 We will prove \cref{prop:op-BM} in \cref{subsec:proof-of-op-BM}.
\begin{example}\label{ex:E0-BM}
    Let $\cO$ be $\EE_0$, which is multi-homwise $(-2)$-connective. Then the result of \cref{prop:op-BM} is a special case of \cref{prop:BM}.
 \end{example}
 \begin{example}\label{ex:Einf-BM}
     Let $\cO$ be the commutative operad $\EE_\infty$, which is multi-homwise $\infty$-connective. Given an $\infty$-topos $\cC$ with objects $X, Y$ in $\CAlg(\cC)  = \Alg_{\EE_\infty}(\cC)$, \cref{prop:op-BM} implies that $X \sqcup Y \to X \times Y$ is an equivalence. This is a special case of \cite[Proposition 3.2.4.7]{HA}.
 \end{example}

\begin{example}\label{ex:BM-and-EHA}
    Let $\cO$ be a reduced coherent $\infty$-operad. Let $\cC$ be an $\infty$-topos $\cC$ with two objects $X, Y \in \Alg_{\cO}(\cC)$. Let $f_1$ and $f_2$ be the unit maps $1_{\cC} \to X$ and  $1_{\cC} \to Y$, which are in general $(-2)$-connected. By \cref{prop:op-BM}, the map 
    \begin{equation}
        X \sqcup Y \to X \times Y
    \end{equation}
    is $d$-connected. This is proven in \cite[Proposition 5.1.3]{EHA}, in fact without the coherent hypothesis on $\cO$. Our proof builds on this base case.
\end{example}

\begin{remark}
    The cubes considered in \cref{prop:op-BM} are strongly coCartesian by \cref{lem:square-strongly-coCart}.
    Note that result of \cref{prop:op-BM} can \emph{not} be extended to general strongly Cartesian functors. For example, if we take $\cO = \EE_\infty$, we are essentially asking for pullback and pushouts to coincide for $\CAlg(\cC)$, which fails as $\CAlg(\cC)$ is not stable. 
\end{remark}

Let us end this subsection with a generalization of \cref{prop:op-BM}:
\begin{corollary}\label{cor:op-BM-k}
    Fix $d \geq -2$ and $-1 \leq m \leq \infty$.
    Let $\cO$ be a coherent multi-homwise $d$-connected $\infty$-operad and $\cC$ be an $m$-topos equipped with the Cartesian symmetric monoidal structure.
    Let $S$ be a finite set and $f_S$ be a collection of $d_i$-connected morphisms $\{f_i \colon X_i \to Y_i\}_{i \in S}$ in $\Alg_{\cO}(\cC)$. 
    Let $\square_{f_S}$ be the functor $\cP(S) \to \cC$ defined in \cref{nota:iterative-square}. For $1 \leq k < n$, the map 
    \begin{equation}
        \colim (\square_{f_S}|_{\leq k}) \to \square_{f_S}(S) = \prod_{i \in S} Y_i
    \end{equation}
    is 
    \begin{equation}\label{eq:minimum}
        \min\left(\sum_{i \in I}(d_i + 2) + k(d+2)-2\right)
    \end{equation}
    -connected. Here we take the minimum over all $I \subset S$ with $|I| = k+1$.
\end{corollary}
\begin{proof}
    We recover \cref{prop:op-BM} as the $k = n-1$ case. Now we induct down on $k$. To simplify the notation let us write $\square_{f_S}$ as $\square$.
    We have a sequence of maps 
    \begin{equation}
        \colim (\square|_{\leq k}) \to \colim (\square|_{\leq (k+1)}) \to \cdots \to \colim (\square|_{\leq n}) = F(S).
    \end{equation}
    By induction, it suffices to show that the first map $\colim (\square|_{\leq k}) \to \colim (\square|_{\leq k+1})$ is
    $$\min_{I \subset S, |I| = K=1}\left(\sum_{i \in I}(d_i + 2) + k(d-2) - 2\right)$$-connected, as the connectivity of the other morphisms are greater.

    Let $\LKan_{k, k+1}\square$ be the left Kan extension of $\square|_{\leq k}$ along the inclusion $\cP_{\leq k}(S) \hookrightarrow \cP_{\leq (k+1)}(S)$. Consider the canonical natural transform $\LKan_{k, k+1}\square \to \square_{\leq (k+1)}$. By construction, this is an equivalence when restricted to $\cP_{\leq k}(S)$. For $I \subset S$ of cardinality $k+1$, we have that
    \begin{equation}
        \LKan_{k, k+1}\square(I) = \colim( \square_{f_I}|_{\leq k})
    \end{equation} is the colimit of $\square_{f_I}$ over the punctured cube $\cP_{\leq k}(I)$. Here $f_I$ is the restriction of $f_S$ to the subset $I$. 
    By \cref{prop:op-BM}, the map 
    \begin{equation}
        \LKan_{k, k+1}\square(I) = \colim \square_{f_I}|_{\leq k} \to 
        \square_{f_I}(I) = 
        \square|_{\leq (k+1)}(I)
    \end{equation}
    is $\left(\sum_{i \in I}(d_i + 2) + k(d-2) - 2\right)$-connected. Therefore the natural transform  $\LKan_{k, k+1}\square \to \square_{\leq (k+1)}$ is pointwise 
    $$\min_{I \subset S, |I| = k+1}\left(\sum_{i \in I}(d_i + 2) + k(d-2) - 2\right)$$-connected. 
    By \cref{obs:lim-colim-preserves}, the class of $k$-connected morphisms is closed under small colimits. It follows that the colimit 
    \begin{equation}
       \colim(\square|_{\leq k}) \simeq \colim(\LKan_{k, k+1}\square) \to \colim(\square|_{\leq (k+1)})
    \end{equation}
    is also 
    $$\min_{I \subset S, |I| = k+1}\left(\sum_{i \in I}(d_i + 2) + k(d-2) - 2\right)$$-connected.
\end{proof}

\subsection{Proof of \cref{prop:op-BM}}\label{subsec:proof-of-op-BM}
This subsection is devoted to the proof of \cref{prop:op-BM}.
First we prove the case for $S = \underline{2}$ and no connectivity assumptions on $f_i$.
This is a generalization of \cite[Proposition 5.1.3]{EHA}:
\begin{proposition}\label{prop:op-BM-base-case-only-d}
    Fix $d \geq -2$ and $-1 \leq m \leq \infty$.
   Let $\cO$ be a reduced (not necessarily coherent) multi-homwise $d$-connected $\infty$-operad and $\cC$ be an $m$-topos equipped with its Cartesian symmetric monoidal structure.  Given a pair of morphisms $f_i \colon X_i \to Y_i$ in $\Alg_{\cO}(\cC)$, 
    the square $f_1 \square f_2$ 
    is $d$-coCartesian.
\end{proposition}
\begin{proof}
    The statement is vacuous for $d = -2$. From now on we assume that $\cO$ is $d$-connected for $d \geq -1$. We are going to show that that $f_1 \boxdot f_2$, i.e. 
    \begin{equation}
        X_1 \times Y_2 \sqcup_{X_1 \times X_2} Y_1 \times X_2 \to Y_1 \times Y_2
    \end{equation}
    is $d$-connected. 
    By \cite[Lemma 5.1.1]{EHA}, the composite 
    \begin{equation}
        X_1 \sqcup Y_2 \to X_1 \times Y_2 \sqcup_{X_1 \times X_2} Y_1 \times X_2 \to Y_1 \times Y_2
    \end{equation}
    has a section; therefore $f_1 \boxdot f_2$ has a section too.
    By \cite[Proposition 4.3.5]{EHA}, it suffices to show that $f_1 \boxdot f_2$ is $(d-\onehalf)$-connected. As the truncation functor $\tau_{\leq d}$ commutes with products and colimits, we see that 
    \begin{equation}
        \tau_{\leq d} (f_1 \boxdot f_2) \simeq (\tau_{\leq d} f_1 \boxdot \tau_{\leq d} f_2).
    \end{equation}
    Hence it suffices to work within the $(d+1)$-category $\cC_{\leq d}$.
    Since the map $\cO \to \EE_\infty$ is a $d$-equivalence, by \cref{obs:d-equivalence-d-cat} we have an equivalence $\Alg_{\cO}(\cC_{\leq d}) \simeq \Alg_{\EE_\infty}(\cC_{\leq d})$. 
    It follows that we are reduced to the case that $\cO = \EE_\infty$. By \cite[Proposition 3.2.4.7]{HA}, the Cartesian product in $\Alg_{\EE_\infty}(\cC)$ can be identified with the categorical coproduct, 
    and the map 
    \begin{equation}
        (X_1 \sqcup Y_2) \sqcup_{X_1 \sqcup X_2} (Y_1 \sqcup X_2) \to Y_1 \sqcup Y_2
    \end{equation}
    is an isomorphism.
\end{proof}
Our next goal is to add the connectivity conditions back on $f_i$.
First we have to adapt the theory of Čech covers, as developed in \cite[\S 6]{HTT}, to our setting.
\begin{proposition}\label{prop:cech-prop}
    Fix $d \geq -2$ and $-1 \leq m \leq \infty$.
    Let $\cO$ be an $\infty$-operad whose underlying $\infty$-category is contractible and $\cC$ be an $m$-topos equipped with the Cartesian symmetric monoidal structure. Given a morphism $f \colon X \to Y$ in $\Alg_{\cO}(\cC)$, the following holds:
    \begin{enumerate}
        \item There exists a simplicial object $F^{\bullet}$ in $\Alg_{\cO}(\cC)$ whose $n$-th term is the iterated fiber product $X \times_{Y} X \cdots \times_{Y} X$ where $Y$ appears $n$ many times.
        \item Let $U \colon \Delta \to \Delta$ be the functor taking $[n]$ to $[n] \sqcup \{\infty\}$ and $F'^{\bullet}$ be the composition $\Deltaop \xrightarrow{U} \Deltaop \xrightarrow{F^\bullet} \cC$. Then the geometric realization of $F'^{\bullet}$ is $X$.
        \item There exists a natural transform $\eta^{\bullet} \colon F'^\bullet \to F^\bullet$. For each $n \in \mbbN$, the map $\eta^n$ has a section $s^n \colon F^{n} \to F'^{n}$.\footnote{Note that the sectios do not form a natural transform as they do not commute with $d_{n-1}$'s.}
        \item If the map $f \colon X \to Y$ is $(-1)$-connected, then the geometric realization of $F^{\bullet}$ is equivalent to $Y$. Furthermore, the geometric realization of $\eta^{\bullet}$ can be identified with $f \colon X \to Y$.
    \end{enumerate}
\end{proposition}
\begin{proof}
    Note that $\Alg_{\cO}(\cC)$ has small limits as it is presentable (\cite[Corollary 3.2.3.5]{HA}). 
    The construction of $F^\bullet$ is given in \cite[Proposition 6.1.2.11]{HTT} as a right Kan extension. This proves part (1). Part (2) follows from \cite[Lemma 6.1.3.17]{HTT}.

    The natural tranform $\eta^\bullet$ comes from the natural transform $\id \to U$ on $\Delta$ given by the inclusion of $[n]$ to $[n] \sqcup \{\infty\}$.
    For each $n$, the section $s^n$ comes from the map $[n] \sqcup \{\infty \} \to [n]$ given by identity on $[n]$ and taking $\{\infty\}$ to the maximal element $n$ of $[n]$. This completes the proof of part (3).

    It remains to prove part (4). 
    First let us consider the case $\cO = \Triv$ where $\Alg_{\cO}(\cC) \simeq \cC$. For $\infty$-topoi, part (4) follows from \cite[Theorem 6.1.0.6(3.iv)]{HTT}. For $m$-topoi with $m < \infty$, part (4) follows from \cite[Theorem 6.4.1.5(6.iv)]{HTT}. 

    For a general $\infty$-operad $\cO$ whose underlying $\infty$-category is contractible, part (4) follows from the fact that the forgetful map $\fgt \colon \Alg_{\cO}(\cC) \to \cC$ preserves limits (\cite[Corollary 3.2.2.4]{HA}), and that it detects and preserves geometric realization (\cite[Proposition 3.2.3.1]{HA}).
    \end{proof}
We also need the following two lemmas:
\begin{proposition}\label{prop:O-alg-stable-under-pullback}
    Fix $d \geq -2$ and $-1 \leq m \leq \infty$.
    Let $\cO$ be a reduced $\infty$-operad and $\cC$ be an $m$-topos equipped with the Cartesian symmetric monoidal structure. The class of $d$-connected morphisms is preserved under pullbacks. Explicitly, for any pullback square 
    \begin{equation}
        \begin{tikzcd}
            B \times_{D} C \ar[r] \ar[d, "f'"] & B \ar[d, "f"]\\ 
            C \ar[r] & D
        \end{tikzcd}
    \end{equation}
    in $\Alg_{\cO}(\cC)$, 
    if $f$ is $d$-connected, then $f'$ is also $d$-connected.
\end{proposition}
\begin{proof}
    Let $\fgt \colon \Alg_{\cO}(\cC) \to \cC$ be the forgetful functor.
    By \cite[Proposition 4.4.5]{EHA}, a morphism in $\Alg_{\cO}(\cC)$ is $d$-connected if and only if the underlying map in $\cC$ is $d$-connected. Furthermore, since the forgetful functor preserves limits, we are reduced to showing that the class of $d$-connected morphisms in $\cC$ is preserved under pullbacks. This is the result of \cite[Proposition 4.8]{gep17}.
\end{proof}
\begin{lemma}\label{lem:section-d-conn}
    Fix $d \geq -2$ and $-1 \leq m \leq \infty$.
    Let $\cO$ be a reduced $\infty$-operad and $\cC$ be an $m$-topos equipped with the Cartesian symmetric monoidal structure. Suppose we have composable maps $A \xrightarrow{f} B \xrightarrow{g} A$ where the composite $g \circ f$ is equivalent to $\id_A$. Then $g$ is $(d+1)$-connected if and only if $f$ is $d$-connected.
\end{lemma}
\begin{proof}
    Once again by \cite[Proposition 4.4.5]{EHA} we are reduced to checking this in $\cC$. As any $m$-topos can be embedded into an $\infty$-topos, so we are reduced to considering the case where $\cC$ is an $\infty$-topos, where it is well-known. See for example \cite[Proposition 4.10.6]{skd}.
\end{proof}

Next we consider the case of $S = \underline{2}$ with connectivity assumptions on $f_i$.
\begin{proposition}\label{prop:op-BM-base-case-general-di}
    Fix $d \geq -2$ and $-1 \leq m \leq \infty$.
    Let $\cO$ be a reduced (not necessarily coherent) multi-homwise $d$-connected $\infty$-operad and $\cC$ an $m$-topos equipped with the Cartesian symmetric monoidal structure.
    For $i \in \{1,2\}$, let $f_i \colon X_i \to Y_i$ be a $d_i$-connected morphism in $\Alg_{\cO}(\cC)$. Then the square 
    $f_1 \square f_2$ is $\left((d_1 + 2) + (d_2+ 2) + d\right)$-connected.
\end{proposition}
\begin{proof}
    Suppose that $d = -2$, then $\cO = \EE_0$ and the result follows from the Blakers-Massey theorem \cref{prop:BM}. Note that the square is strongly Cartesian by \cref{lem:square-strongly-coCart}.
    From now on we will assume that $\cO$ is $d$-connected for $d \geq -1$.

    We induct on $d_1$ while fixing $d_2 =-2$. Note that the base case of $d_1 = d_2 = -2$ is proven in \cref{prop:op-BM-base-case-only-d}. 
    By \cref{prop:cech-prop}, we have a natural transform of simplicial objects  $\eta^{\bullet} \colon F'^{\bullet} \to F^{\bullet}$ in $\Alg_{\cO}(\cC)$ whose geometric realization is $f_1 \colon X_1 \to Y_1$. Hence 
    $\eta^{\bullet} \square f_2$ gives a simplicial diagram of squares 
    \begin{equation}\label{eq:cech-squares}
        \begin{tikzcd}
            F'^n \times X_2 \ar[r] \ar[d,  "\eta^n"]& F'^n \times Y_2 \ar[d, "\eta^n"] \\ 
            F^n \times X_2 \ar[r] & F^n \times Y_2
        \end{tikzcd}
    \end{equation}
    whose geometric realization recovers $f_1 \square f_2$. Since $d$-coCartesian squares are stable under colimits (\cref{obs:d-coCart-colimit}), it suffices to show that for each $n$, the square \eqref{eq:cech-squares} is $((d_1 + 2) + d)$-coCartesian.
    Consider the following commutative diagram:
    \begin{equation}
        \begin{tikzcd}
            F^n \times X_2 \ar[r] \ar[d, "s^n"] & F^n \times Y_2 \ar[d, "s^n"]\\
            F'^n \times X_2 \ar[r] \ar[d,  "\eta^n"]& F'^n \times Y_2 \ar[d, "\eta^n"] \\ 
            F^n \times X_2 \ar[r] & F^n \times Y_2
        \end{tikzcd}
    \end{equation}
    where $s^n$ is the section of $\eta^n$ given in \cref{prop:cech-prop}. The total square is $\id \square f_2$, which is $\infty$-coCartesian. By \cref{lem:d-coCart-section-lemma}, to show that the bottom square $\eta^n \square f_2$ is $((d_1 + 2) + d)$-connective, it suffices to show that the top square $s^n \square f_2$ is $((d_1 + 1) + d)$-connective. Being a pullback of $f_1$, the map $\eta^n$ is also $d_1$-connected by \cref{prop:O-alg-stable-under-pullback}.
    Being a section of $\eta^n$, the map $s^n$ is $(d_1 -1)$-connected by \cref{lem:section-d-conn}. By induction, it follows that $s^n \square f_2$ is $((d_1 + 1) + d)$-connective. This completes the induction on $d_1$. To complete the proof we fix $d_1$ and use the same argument on $d_2$.
\end{proof}
\begin{remark}
    Up to this point we have not assumed that $\cO$ is coherent. The coherence condition is needed for the inducting on the dimension of the cubes. 
\end{remark}
Finally we are ready to prove \cref{prop:op-BM}:
\begin{proof}[Proof of \cref{prop:op-BM}]
    Pick an isomorphism $S \simeq \underline{n}$. The base case of $n = 2$ is proven in \cref{prop:op-BM-base-case-general-di}. Now we induct on $n$. Since $\cO$ is coherent, by \cref{prop:O-alg-preserves-pushout} and \cref{prop:iter-colim}, the map
    $$\colim (\cF_{f_{\underline{n}}}|_{\leq k}) \to \cF_{f_{\underline{n}}}(\underline{n}) = \prod_{i=1}^n Y_i$$
is equivalent to $$((f_1 \boxdot f_2) \boxdot \cdots) \boxdot f_n.$$
By induction, the map $((f_1 \boxdot f_2) \boxdot \cdots) \boxdot f_{n-1}$ is 
$$
\left(\sum_{i =1}^{n-1} (d_i + 2) + (n-1)(d+2) - 2\right)$$
-connected. 
By \cref{prop:op-BM-base-case-general-di}, it follows that $((f_1 \boxdot f_2) \boxdot \cdots) \boxdot f_n$ is 
\begin{equation}
    \left(\sum_{i =1}^{n-1} (d_i + 2) + (n-1)(d+2) - 2 + 2\right) + (d_n + 2) +  d = \sum_{i =1}^{n}  (d_i + 2) + n(d+2) - 2
\end{equation}
-connected.
\end{proof}

\section{Generalized Eckmann-Hilton argument}\label{sec:gen-EHA}
In this section we prove  \cref{thm:main-thm}, which is a Eckmann-Hilton type argument that incorporates arities of $\infty$-operads. After reviewing the theory of unital $k$-restricted $\infty$-operads developed as \cite{dubey2024unital} in  \cref{subsec:arity}, we study the notion of $k$-wise $d$-connected $\infty$-operads in \cref{subsec:BM-for-alg}. Lastly, we prove \cref{thm:main-thm} in \cref{subsec:gen-EHA}.
\subsection{Arity restricted $\infty$-operads}\label{subsec:arity}
Let $\Opun$ denote the $\infty$-category of unital $\infty$-operads, and for $k \geq 1$, let $\Opunleqk$ be the $\infty$-category of unital $k$-restricted $\infty$-operads (\cite[Definition 3.21]{dubey2024unital}). The main result of \cite{dubey2024unital} is the following:
\begin{proposition}[{\cite[Theorem 1.2]{dubey2024unital}}]
   Given $1 \leq k$,  there is a diagram of adjunctions
   \begin{equation}
   \begin{tikzcd}[column sep=1.5cm]
\Opunleqk
\arrow[bend left=45, hookrightarrow]{r}{\Lk}
\arrow[leftarrow]{r}[xshift=-0.15cm, yshift=0.2cm]{\bot}[swap, xshift=-0.15cm, yshift=-0.2cm]{\bot}[description, xshift=-0.15cm]{(-)^k}
\arrow[bend right=45, hookrightarrow]{r}[swap]{\Rk}
&
\Opun
\end{tikzcd}
\end{equation}
where the middle map $(-)^k$ is the restriction functor, and the left and right adjoints $\Lk$ and $\Rk$ are fully faithful.
\end{proposition}
\begin{remark}
    Fix $k \geq 1$.
    Let $\cO$ be a unital $k$-restricted $\infty$-operad. Intuitively, for $n > k$, the $n$-ary space of $\Lk \cO$ is the space of all possible $n$-ary morphisms that can be constructed from the $\leq k$-ary morphisms of $\cO$. On the other hand, the $n$-ary space of  $\Rk \cO$ is the space of collections of $\leq k$-ary morphisms, one for each subset $I$ of $\underline{n}$ with cardinality $\leq k$, that are compatible under taking units. See \cite[Corollaries 4.11, 5.18]{dubey2024unital}.
\end{remark}
The composite $\Lk \circ (-)^k$ is a colocalization functor on $\Opun$, which we will simply denote as $\Lk$. Similarly, the composite $\Rk \circ (-)^k$ is a localization functor, which we will denote as $\Rk$.

Let $\cC$ be a unital symmetric mononidal $\infty$-category. Given objects $X_1, \cdots, X_n$ in $\cC$, by \cref{nota:iterative-square}, we have a $n$-cube $\square_{X_{\underline{n}}} \colon \cP(\underline{n}) \to \cC$. 
The following proposition gives a colimit description of the multi-ary hom spaces of $\Rk \cC$:

\begin{corollary}\label{cor:RkcC-colimit-description}
    Let $\cC$ be a unital symmetric monoidal $\infty$-category. 
     Given objects $X_1, \cdots, X_n$ in $\cC$ and $k \geq 1$, if the colimit of $\square_{X_{\underline{n}}}|_{\leq k}$ exists, then for any $Y \in \cC$ we have an equivalence: 
     \begin{equation}
        \Mul_{\Rk \cC}(X_1, \cdots, X_n; Y) \simeq \Hom_{\cC}(\colim \square_{X_{\underline{n}}}|_{\leq k}, Y).
     \end{equation}
\end{corollary}
\begin{proof}
    This follows directly from the description of the $n$-ary morphism spaces of $\Rk \cC$ in \cite[Corollary 5.18]{dubey2024unital}.
\end{proof}

\begin{remark}
    Suppose $\cC$ has all small colimits; then 
    the statement of \cref{cor:RkcC-colimit-description} implies that the fibration $(\Rk \cC)^{\otimes} \to \Finstar^{\otimes}$ is locally coCartesian. However, this does \emph{not} imply that the fibration is coCartesian, i.e. $\Rk \cC$ defines a symmetric monoidal structure on $\cC$.
\end{remark}
\subsection{$k$-wise $n$-connected $\infty$-operads}
In this section we define $k$-wise $d$-connected $\infty$-operads.
\begin{definition}\label{def:k-wise-d-connected}
   Fix $k \geq 1$, $d \geq -2$, and $-1 \leq m \leq \infty$. A reduced $\infty$-operad $\cO$ is called \emph{$k$-wise $d$-connected} if for any $m$-topos $\cC$,  equipped with its Cartesian presentably symmetric monoidal structure, 
   and for any list of $X_1, \cdots, X_n$ in $\Alg_{\cO}(\cC)$, the map 
   \begin{equation}
    \colim(\square_{X_{\underline{n}}}|_{\leq k}) \to \prod_{i = 1}^{n} X_i
   \end{equation}
   in $\Alg_{\cO}(\cC)$ is $d$-connected. Here $\square_{X_{\underline{n}}}$ is the $n$-cube defined in \cref{nota:iterative-square}.
\end{definition}
\cref{cor:op-BM-k} gives a $k$-wise $d$-connected bound for coherent $\infty$-operads:
\begin{corollary}\label{cor:cohernet-d-connected}
   Let $\cO$ be a coherent multi-homwise $d$-connected $\infty$-operad; then $\cO$ is $k$-wise $k(d+2)-2$-connected.  
\end{corollary}
\begin{example}\label{ex:En-k-wise-connected}
    By \cref{ex:connected-level} and \cref{prop:Ek-coherent}, for $n \geq 0$, the little cubes operad $\EE_n$ is $k$-wise $(kn-2)$-connected.
\end{example}
\begin{observation}\label{obs:E_0-k-wise-d-conn}
    Note that $\EE_0$ is not $k$-wise $d$-connected for any $d \geq -1$.
\end{observation}
The next proposition allows us to give $k$-wise $d$-connected bounds for non-coherent $\infty$-operads, in particular for $\AA_k$ (see \cref{rem:Ak-not-coherent}).

\begin{proposition}\label{prop:k-wise-d-conn-and-d-eq}
    Let $f \colon \cO_1 \to \cO_2$ be a map of reduced $\infty$-operads. If $\cO_2$ is $k$-wise multi-homwise $d$-connected and $f$ is a $d$-equivalence (\cref{def:op-d-equivalence}), then $\cO_1$ is also $k$-wise $d$-connected.
\end{proposition}

\begin{proof}
    The statement is vacuous when $d = -2$. 
    From now on we assume that $d > -2$. By \cref{obs:E_0-k-wise-d-conn}, we see that $\cO_1$ and $\cO_2$ are not $\EE_0$. 
    Given $n$ objects $X_1, \cdots, X_n$ in $\Alg_{\cO}(\cC)$,
    by the same argument as the proof of  \cref{prop:op-BM-base-case-only-d}, the map $\colim(\square_{X_{\underline{n}}}|_{\leq k}) \to \prod X_i$ has a section. By \cite[Prosition 4.3.5]{EHA}, it suffices to show that this map is $(d - \onehalf)$-connected. 
    As the truncation functor $\tau_{\leq d}$ commutes with products and colimits, 
    the $d$-truncation of $\colim(\square_{X_{\underline{n}}}|_{\leq k}) \to \prod X_i$ is
    \begin{equation}
      \colim(\square_{\tau_{\leq d}X_{\underline{n}}}|_{\leq k}) \to \prod \tau_{\leq d} X_i.
    \end{equation}    
    Hence we are reduced to considering the statement for $\cC_{\leq d}$, which is a $(d+1)$-category. By \cref{obs:d-equivalence-d-cat}, the map $\Alg_{\cO_2}(\cC_{\leq d}) \to \Alg_{\cO_1}(\cC_{\leq d})$ is an equivalence of $\infty$-categories, and the result follows from the fact that $\cO_2$ is $k$-wise $d$-connected.
\end{proof}

Recall from \cref{ex:connected-level} that $\AA_n \to \EE_1$ is multi-homwise $(n-3)$-connected. Finally we get the $k$-wise $d$-connected bounds for $\AA_n$.
\begin{corollary}\label{cor:AAn-k-wise-d-conn}
    The $\infty$-operad $\AA_n$ is $k$-wise $\min(n-3, k-2)$-connected.
\end{corollary}

The notion of $k$-wise $d$-connected $\infty$-operads allow us to get bounds for the truncated level of $\cC \to \Rk\cC$:
\begin{proposition}\label{prop:AlgO-RkAlgO-truncated}
    Fix $d, m \geq -2$ and $k \geq 1$.
    Let $\cO$ be a reduced $\infty$-operad and $\cC$ be a $(m+1)$-topos equipped with the Cartesian symmetric monoidal structure. 
    Suppose that $\cO$ is $k$-wise $d$-connected, then the map of $\infty$-operads
    \begin{equation}
        \Alg_{\cO}(\cC) \to  \Rk(\Alg_{\cO}(\cC))
    \end{equation}
    is multi-homwise $(m-d-2)$-truncated.
\end{proposition}
\begin{proof}
    Given objects $X_1, \cdots, X_n, Y$ in $\cC$, by \cref{cor:RkcC-colimit-description}, we have an equivalence 
    \begin{equation}
        \Mul_{\Rk\Alg_{\cO}(\cC)}(X_1, \cdots, X_n; Y) \simeq \Hom_{\Alg_{\cO}(\cC)}(\colim(\square_{X_{\underline{n}}}|_{\leq k}), Y).
    \end{equation}
    We would like to show that the map
    \begin{equation}\label{eq:helper-1}
        \Hom_{\Alg_{\cO}(\cC)}(\prod X_i, Y) \to \Hom_{\Alg_{\cO}(\cC)}(\colim(\square_{X_{\underline{n}}}|_{\leq k}), Y)
    \end{equation}
    is $(n-d-2)$-truncated. It suffices to show that the fibers are $(n-d-2)$-truncated. Fix a map $f \colon \colim(\square_{X_{\underline{n}}}|_{\leq k}) \to Y$.
    Consider the square 
    \begin{equation}
        \begin{tikzcd}
            \colim(\square_{X_{\underline{n}}}|_{\leq k}) \ar[r, "f"] \ar[d] & Y \ar[d]\\ 
            \prod X_i \ar[r] & * 
        \end{tikzcd}
    \end{equation}The space of lifts of this square is precisely the fiber of \eqref{eq:helper-1} at $f$.
    By assumption the left vertical map is $d$-connected (as $\cO$ is a $k$-wise $d$-connected $\infty$-operad) and the right vertical map is $m$-truncated (as $\cC$ is a $(m+1)$-category). By \cite[Proposition 4.2.8]{EHA}, the space of lifts is 
    $(m-d-2)$-truncated.
\end{proof}

\subsection{Generalized Eckmann-Hilton argumet}\label{subsec:gen-EHA}
In this subsection we prove our main result, \cref{thm:main-thm}.
First we have the following observation:
\begin{observation}\label{obs:adj-obs}
    Consider an adjunction
    $
        \begin{tikzcd}[column sep=1.5cm]
    F  \colon \cC
    \arrow[right, yshift=0.9ex]{r}{}
    \arrow[leftarrow, yshift=-0.9ex]{r}[yshift=-0.15ex]{\bot}[swap]{}
    &
    \cD \colon G
    \end{tikzcd}$
    between $\infty$-categories and objects $c, d \in \cC$. Suppose we have a morphism $g \colon Fc \to Fd$ in $D$, then the 
    space of fibers 
    \begin{equation}
        \Hom_{\cC}(c,d) \times_{\Hom_{\cD}(Fc, Fd)} \{g\}
    \end{equation}
    is equivalent to the
    space of lifts 
    \begin{equation}
        \begin{tikzcd}
             & d \ar[d] \\ 
        c \ar[ur, dashed] \ar[r, "g'"] & GFd
        \end{tikzcd}
    \end{equation}
    where $g' \colon c \to GFd$ is the adjunt of $Fg$.
\end{observation}
This implies the following useful corollary:
\begin{lemma}\label{lem:adjunction-section-lem}
    Consider an adjunction
    $
        \begin{tikzcd}[column sep=1.5cm]
    F  \colon \cC
    \arrow[right, yshift=0.9ex]{r}{}
    \arrow[leftarrow, yshift=-0.9ex]{r}[yshift=-0.15ex]{\bot}[swap]{}
    &
    \cD \colon G
    \end{tikzcd}$
    between $\infty$-categories and objects $b, c, d \in \cC$. 
    Suppose we have morphisms $f \colon b \to c$ and $g \colon b \to d$ in $\cC$ such that $Ff \colon Fb \to Fc$ is an equivalence. Let $\phi \colon c \to GFd$ denote the adjunct to the composite $Fc \xrightarrow{Ff^{-1}} Fb \xrightarrow{Fg} Fd$. Let $L_1$ and $L_2$ denote the space of lifts of the left and right diagrams respectively:
    \begin{equation}
        \begin{tikzcd}
            b \ar[r] \ar[d] & d \ar[d] \\ 
            c \ar[ur, dashed] \ar[r, "\phi"] & GFd
        \end{tikzcd}
        \quad
        \quad
        \begin{tikzcd}
            b \ar[r] \ar[d] & d \\ 
            c \ar[ur, dashed]
        \end{tikzcd} 
    \end{equation} 
    Then the forgetful map $L_1 \to L_2$ has a section.
\end{lemma}
\begin{proof}
    Given a map $\tilde{g} \colon c \to d$ lifting $b \to d$, the induced map $F\tilde{g} \colon Fc \to Fd$ is canonically equivalent to $\phi$. Now the result follows from \cref{obs:adj-obs}.
\end{proof}

Now we are ready to state the generalized Eckmann-Hilton argument:
\begin{theorem}\label{thm:main-thm}
    Fix $d_1, d_2 \geq -2$ and $k \geq 1$.
    Let $f \colon \cP \to \cQ$ be a multi-homwise $d_1$-connected map between reduced $\infty$-operads such that the $k$-restriction $f^k \colon \cP^k \to \cQ^k$ is an equivalence. Let $\cR$ be a reduced $k$-wise $d_2$-connected $\infty$-operad, then the induced map 
    \begin{equation}
        \cP \otimes \cR \to \cQ \otimes \cR
    \end{equation}
    is a $(d_1 + d_2 + 2)$-equivalence.
\end{theorem}

\begin{proof}
Let $D$ be $d_1 + d_2 + 2$. By \cite[Proposition 3.2.6]{EHA}, it suffices to show that for any $(D+1)$-topos $\cC$, the map of spaces 
\begin{equation}
    \Hom_{\Op}(\cQ, \Alg_{\cR}(\cC)) \simeq \Hom_{\Op}(\cQ \otimes \cR, \cC) \to 
    \Hom_{\Op}(\cP \otimes \cR, \cC) 
     \simeq 
     \Hom_{\Op}(\cP, \Alg_{\cR}(\cC))
\end{equation}
is an equivalence. Equivalently, it suffices to show that for any diagram of the form 
\begin{equation}
    \begin{tikzcd}
        \cQ \ar[r] \ar[d, "f"] & \Alg_{\cR}(\cC) \\ 
        \cP \ar[ur, dashed] & 
    \end{tikzcd}
\end{equation}
the space of lift is contractible. 
As a section of a contractible space is also contractible, 
by \cref{lem:adjunction-section-lem}, it suffices to show that the square  
\begin{equation}\label{eq:main-theorem-diag-2}
    \begin{tikzcd}
        \cQ \ar[r] \ar[d, "f"] & \Alg_{\cR}(\cC) \ar[d]\\ 
        \cP \ar[ur, dashed] \ar[r]& \Rk(\Alg_{\cR}(\cC))
    \end{tikzcd}
\end{equation}
has a contractible space of lifts.
Note that the left vertical map is multi-homwise $d_1$-connected by hypothesis and surjective on colors (\cref{ex:reduced-surjective-on-colors}), while the right vertical map is multi-homwise $(D - d_2 -2) = d_1$-truncated by \cref{prop:AlgO-RkAlgO-truncated}. As the classes of (surjective on colors and multi-homwise $d_1$-connected, multi-homwise $d_1$-truncated) morphisms are left and right classes of a factorization system on $\Op$ (\cref{obs:op-fact-sys}), it follows that the space of lift of \eqref{eq:main-theorem-diag-2} is contractible.
\end{proof}
By \cref{cor:cohernet-d-connected}, we get the following:
\begin{corollary}\label{cor:main-thm-coherent}
    Fix $d_1, d_2 \geq -2$ and $k \geq 1$.
    Let $f \colon \cP \to \cQ$ be a multi-homwise $d_1$-connected map between reduced $\infty$-operads such that the $k$-restriction $f^k \colon \cP^k \to \cQ^k$ is an equivalence. Let $\cR$ be a coherent multi-homwise $d_2$-connected $\infty$-operad; then the induced map 
    \begin{equation}
        \cP \otimes \cR \to \cQ \otimes \cR
    \end{equation}
    is a $(d_1 + k (d_2 + 2))$-equivalence.
\end{corollary}

\begin{remark}
    Note that any map of reduced $\infty$-operads induces an equivalence on $1$-restrictions, in which case the connectivity bound agrees with \cite[Theorem 5.3.1]{EHA}.
    Note that we asked for $\cP \to \cQ$ to be $d$-connective rather than just a $d$-equivalence. This is because we are lifting against $\cC \to \Rk \cC$ rather than lifting against certain endomorphism $\infty$-operads over $\EE_\infty$.
\end{remark}

\section{$\EE_n$ algebras in $m$-categories}\label{sec:En-in-m-cat}
In this section we apply \cref{thm:main-thm} to the problem of approximating $\AA_{k_1} \otimes \AA_{k_2} \otimes \cdots \otimes \AA_{k_n} \to \EE_n$. 
First we apply \cref{thm:main-thm} to the morphism $\AA_k \to \AA_j$. Recall that the map $\AA_k \to \AA_j$ is multi-homwise $(k-3)$-connected (\cref{ex:connected-level}) and induces an isomorphism $(\AA_k)^k \simeq (\AA_j)^k$:
\begin{corollary}\label{cor:Ak-Aj-otimes-R-equiv}
    Fix $1 \leq k < j \leq \infty$ and $d \geq -2$. Let $\cR$ be a reduced $k$-wise $d$-connected $\infty$-operad. Then the map of $\infty$-operads
    \begin{equation}
        \AA_k \otimes \cR \to \AA_j \otimes \cR
    \end{equation}
    is a $(k + d -1)$-equivalence. 
\end{corollary}
We are ready to now state our main theorem: 
\begin{theorem}\label{thm:AAk-equivalence}
    Fix $1 \leq k_1, k_2, \cdots, k_n \leq \infty$. Let $k_{\min}$ be $\min(k_1, k_2, \cdots, k_n)$ and $r$ be the number of occurrences of $k_{\min}$ among $k_1, \cdots, k_n$. The following holds:
    \begin{enumerate}
        \item For any $1 \leq i \leq n$ and $j_i$ such that $k_i < j_i \leq \infty$, the map 
        \begin{equation}
            \AA_{k_1}\otimes  \cdots \AA_{k_{i-1}} \otimes \AA_{k_{i}} \otimes \AA_{k_{i+1}} \cdots  \otimes \AA_{k_n} 
            \to 
            \AA_{k_1}  \otimes \cdots \AA_{k_{i-1}} \otimes \AA_{j_i} \otimes \AA_{k_{i+1}} \cdots \otimes \AA_{k_n} 
        \end{equation}
        is a \begin{equation}\label{eq:second-last}
            \begin{cases}
                k_i + (n-1)k_{\min} - 2 - r, & k_i > k_{\min} \\ 
                nk_{\min} -2 - r, & k_i = k_{\min}
            \end{cases}
        \end{equation}
        -equivalence. 
        \item The map of $\infty$-operads
        \begin{equation}
            \AA_{k_1} \otimes \AA_{k_2} \otimes \cdots \otimes \AA_{k_n} \to \EE_n
        \end{equation}
        is a $(nk_{\min} - 2 - r)$-equivalence.
        \item For $k \geq 1$, the $\infty$-operad $\AA_{k_1} \otimes \AA_{k_2} \otimes \cdots \otimes \AA_{k_n}$ is $k$-wise $\min(nk_{\min} - 2 - r, nk-2)$-connected.
    \end{enumerate}
\end{theorem}
\begin{proof}
    Part (1) implies part (2) by iteratively taking $j_i$ to be $\infty$, and part (2) implies part (3) by  \cref{prop:k-wise-d-conn-and-d-eq} and \cref{ex:En-k-wise-connected}.

    It remains to prove part (1).
    The base case of $n = 1$ is covered in \cref{ex:connected-level}. Now we induct on $n$. Without loss of generality, we can assume that $i = 1$.
    By induction, the $\infty$-operad
    \begin{equation}
        \AA_{k_2} \otimes \AA_{k_3} \otimes \cdots \otimes \AA_{k_n}
    \end{equation}
    is $k$-wise $\min((n-1)k'_{\min} - 2 - r', (n-1)k-2)$-connected. Here $k'_{\min}$ is  $\min(k_2, \cdots, k_n)$ and $r'$ is the number of  occurrences of $k'_{\min}$ among $k_2, \cdots, k_n$.
    By \cref{cor:Ak-Aj-otimes-R-equiv}, the map 
    \begin{equation}\label{eq:last-helper-map}
        \AA_{k_1} \otimes (\AA_{k_2} \otimes \cdots \otimes \AA_{k_n}) \to \AA_{j_1} \otimes (\AA_{k_2} \otimes \cdots \otimes \AA_{k_n})
    \end{equation}
    is a \begin{equation}\label{eq:last-helper}
        k_1 + \min((n-1)k'_{\min} - 2 - r', (n-1)k_1-2) -1
    \end{equation}-equivalence. It suffices to check that \eqref{eq:last-helper} is equal to \eqref{eq:second-last}.
    
    If $k = k_{\min}$ and $k_{\min} < k'_{\min}$, then $r = 1$ and \eqref{eq:last-helper} is a $(nk_{\min} - 2 - r)$-equivalence. 
    Similarly, if $k = k_{\min} = k'_{\min}$, then $r = r'+1$ and \eqref{eq:last-helper} is a $(nk_{\min} - 2 - r)$-equivalence. Lastly, if $k > k_{\min}$, then $k_{\min} = k'_{\min}$, $r = r'$, and \eqref{eq:last-helper} is a $(k_1 +(n-1)k_{\min} - 2 - r)$-equivalence.
\end{proof}
Combine \cref{thm:AAk-equivalence}(2) and \cref{obs:d-equivalence-d-cat}, we get the following:
\begin{corollary}[\cref{cor:intro-main-cor}]\label{cor:En-alg-in-m-cat}
    Fix $1 \leq n \leq \infty$ and $m \geq -1$. Let $\cC$ be a symmetric monoidal $m$-category. The restriction functor 
    \begin{equation}
        \Alg_{\EE_n}(\cC) \to \Alg_{\AA_{k+1}^{\otimes(n-s)} \otimes \AA_{k}^{\otimes s}}(\cC)
    \end{equation}
    is an equivalence of $\infty$-categories, where $k = \lceil \frac{m+1}{n} \rceil + 1$ and $s$ is  $(nk - m-1)$.
\end{corollary}
\begin{example}
    Take $n = 1$. Then \cref{cor:En-alg-in-m-cat} recovers the result that an $\EE_1$-algebra in an $m$-category is an $\AA_{m+2}$-algebra.
\end{example}

We end by going through the implications of \cref{cor:En-alg-in-m-cat} for small $m$:
\begin{example}[$\EE_n$-algebras in $(-2)$-categories]
   The only $(-2)$-category is the terminal category $\pt$. Note that $\Alg_{\cO}(\pt) \simeq \pt$ for any $\infty$-operad $\cO$. 
\end{example}
    
\begin{example}[$\EE_n$-algebras in $(-1)$-categories]
   There are only two $(-1)$-categories, namely $\varnothing$ and $\pt$. In this case an $\EE_n$-algebra structure on a $(-1)$-category $\cC$ is simply equivalent to an object $c \in \cC$, equivalently an $\AA_1$-algebra. Thus $\Alg_{\EE_n}(\cC) \simeq \Alg_{\AA_1}(\cC)$. This agrees with \cref{cor:En-alg-in-m-cat}.
\end{example}

\begin{example}[$\EE_n$-algebras in $0$-categories]
    A $0$-category is a poset, and a symmetric monoidal $0$-category is a poset $\cC$ together with a unit $1_{\cC}$ and a unital symmetric product $\times \colon \cC \times \cC \to \cC$ that preserves the poset structure.
    For any object $c \in \cC$ and $n \geq 1$, there exists an $\EE_n$-algebra structure on $c$ if and only if $c \times c \leq c$ and $1_{\cC} \leq c$. Furthermore, such a structure is unique if it exists. It follows that an $\EE_n$-algebra structure on $c$ is equivalent to an $\AA_2$-algebra structure on $c$, as predicted by \cref{cor:En-alg-in-m-cat}.
\end{example}
\begin{example}[$\EE_n$-algebras in $1$-categories]\label{ex:1-cat}
    An $1$-category is an ordinary category.
    For an object in a symmetric monoidal $1$-category, an 
    $\EE_1$-algebra structure is equivalent to  an $\AA_3$-algebras structure, which unpacks the classical notion of associative algebra structure.
    
    The $n \geq 2$ cases are more interesting. Let $c$ be an object in a symmetric monoidal $1$-category $\cC$.
    The classical Eckmann-Hilton argument \cite{EH} states that any pairs of unital binary operations on $c$ that distributes over each other are equal and commutative.
    This 
    implies that for any $2 \leq n \leq \infty$, an $\EE_n$-algebra structure on $c$ is equivalent to an $\AA_2 \otimes \AA_2$-algebra struture on $c$, agreeing with \cref{cor:En-alg-in-m-cat}.
\end{example}
\begin{example}[$\EE_n$-algebras in $2$-categories]\label{ex:2-cat}
    Let us consider the symmetric monoidal $2$-category $\Cat$ of $1$-categories.
    An $\EE_1$-algebra in $\Cat$ is a monoidal category, which is a category with a unital multiplication and an associator that satisfies the pentagon axiom. 
    Such structure is precisely an $\AA_4$-algebra structure. 

    More interestingly, it is a well-known folklore  that 
    an $\EE_2$-algebra in $\Cat$ is a braided monoidal category. It turns out this equivalence actually factors through $\AA_2 \otimes \AA_3$:
    the braided monoidal categories are naturally $\AA_2 \otimes \AA_3$-algebras (see \cite[Example 5.1.2.4]{HA}), and \cref{cor:En-alg-in-m-cat} implies that $\AA_2 \otimes \AA_3$-algebras in  symmetric monoidal $2$-categories are equivalent to $\EE_2$-algebras.\footnote{A similar factoring is done in \cite{braiding}. In \cite[Corollary 7.4.15(2)]{braiding} it is shown that  braided monoidal categories are $\AA_2 \otimes \EE_1$-algebras in $\Cat$.  Furthermore, in \cite[Corollary 7.7.8]{braiding} it is shown that 
    $\AA_2 \otimes \EE_1$-algebra in symmetric monoidal $2$-categories are equivalent to $\EE_2$-algebras. Note that we have shown that $\AA_2 \otimes \EE_1$, $\AA_2 \otimes \EE_1$, and $\EE_2$ are all $1$-equivalent to each other.}

    For $n \geq 3$, $\EE_n$-algebras in $\Cat$ are equivalent to $\EE_\infty$-algebras, which are  symmetric monoidal categories. This is reflected in \cref{cor:AAn-k-wise-d-conn} by the statement that $\EE_n$-algebras in symmetric monoidal $2$-categories for $n \geq 3$ are equivalent to $\AA_2^{\otimes 3}$-algebras.
\end{example}

\begin{remark}\label{rem:three-cat-case}
    Let $\Cat_{2}$ be the $3$-category of $2$-categories.
    There has been a long history for finding the right definition of a braided monoidal $2$-category. We refer the reader to \cite{schommer-pries-thesis} for the definition of a braided monoidal $2$-category, as well as the history of finding the right definition.
    It would be interesting to check that a braided monoidal $2$-category (viewed as an object in $\Cat_{2}$) is equivalent to an $\AA_3 \otimes \AA_3$-algebra in $\Cat_{2}$, which by \cref{cor:AAn-k-wise-d-conn} would imply that it is also equivalent to an $\EE_2$-algebra in $\Cat_{2}$.
    
    Similarly, it would be interesting to check that a sylleptic monoidal $2$-category is equivalent to an $\AA_3 \otimes \AA_2 \otimes \AA_2$-algebra  in $\Cat_2$, which is equivalent to an $\EE_3$-algebra in $\Cat_2$ by \cref{cor:AAn-k-wise-d-conn}.
\end{remark}

\bibliographystyle{amsalpha}
\bibliography{bib.bib}

\end{document}